\documentclass[12pt, a4paper]{article}
\usepackage{amsmath, amsfonts, amssymb, dsfont, upref,amsthm,eucal}

\usepackage{bm}
\usepackage{graphicx}
\usepackage{graphics}
\usepackage[swedish,english]{babel}
\usepackage[latin1]{inputenc}

\usepackage{mathrsfs}

\numberwithin{equation}{section}
\newtheorem{theorem}{Theorem}
\newtheorem{lemma}[theorem]{Lemma}
\newtheorem{remark}[theorem]{Remark}
\newtheorem{proposition}[theorem]{Proposition}
\newtheorem{corollary}[theorem]{Corollary}
\numberwithin{theorem}{section}
\DeclareMathOperator*{\argmax}{arg\,max}

\newtheorem{definition}[theorem]{Definition}
\theoremstyle{remark}

\newcommand\Atop[2]{\genfrac{}{}{0pt}{}{#1}{#2}}
\newcommand{\Ab}{\mathbf{A}}
\newcommand{\AF}{\mathfrak{A}}
\newcommand{\AK}{{\bf\mathcal{A}}}
\newcommand{\EE}[1]{\mathds{E}\left[#1\right]}
\newcommand{\EXP}{\mathds{E}}
\newcommand{\fracN}{{\tfrac{1}{N}}}
\newcommand{\MM}{\mathscr{M}}
\newcommand{\NN}{\mathds{N}}
\newcommand{\norm}[1]{\left\vert#1\right\vert}
\newcommand{\Norm}[1]{\left\Vert#1\right\Vert}
\newcommand{\prob}{\mathds{P}}

\newcommand{\pset}{\mathscr{P}}
\newcommand{\RR}{\mathds{R}}
\newcommand{\XX}{\mathscr{X}}

\newcommand{\Pmass}{\mathscr{P}}
\newcommand{\deltab}{\bm{\delta}}
\newcommand{\mub}{\bm{\mu}}
\newcommand{\etab}{\bm{\eta}}

\newcommand{\ub}{\bm{u}}

\newcommand{\xb}{\bm{x}}

\newcommand{\prtl}{\partial}

\date{}
\begin{document}

\title{An Approximate Nash Equilibrium for Pure Jump Markov Games of Mean-field-type on Continuous State Space}
\author{Rani Basna\thanks{Department of Mathematics, Linnaeus University, V\"axj\"o, 351 95, Sweden. e-mail: rani.basna@lun.se}\,\,  \, Astrid Hilbert\thanks{Department of Mathematics, Linnaeus University, V\"axj\"o, 351 95, Sweden. e-mail: astrid.hilbert@lnu.se}\,\, and \, Vassili N. Kolokoltsov\thanks{ Department of Statistics, University of Warwick, Coventry, CV4 7AL, UK. e-mail: v.kolokoltsov@warwick.ac.uk}}
\maketitle

\begin{abstract}
We investigate mean-field games from the point of view of a large number of indistinguishable players, which eventually converges to infinity. The players are weakly coupled via their empirical measure. The dynamics of the states of the individual players is governed by a non-autonomous pure jump type semi group in a Euclidean space, which is not necessarily smoothing. Investigations are conducted in the framework of non-linear Markov processes. We show that the individual optimal strategy results from a consistent coupling of an optimal control problem with a forward non-autonomous dynamics. In the limit as the number $N$ of players goes to infinity this leads to a jump-type analog of the  well-known non-linear McKean-Vlasov dynamics. The case where one player has an individual preference different from the ones of the remaining players is also covered. The two results combined reveal an epsilon-Nash Equilibrium for the $N$-player games.
\end{abstract}

60J75, 60J25.

{$\bf Keywords$}: Mean-field games, pure jump Markov process, Dynamic programing,
Optimal control, Non-linear Markov processes, Koopman Dynamics, $\epsilon$-Nash equilibrium.

\parindent = 0pt
\section{Introduction}

Mean-field game theory is a type of dynamic Game theory. The indistinguishable individual agents are coupled with each other by their individual dynamics through the empirical measure. Moreover, the objective for each player is a function not only of her own preference and decision but also of the decisions of the other players via the mean-field. Mathematically seen it is a combination between mean-field theory and the theory of stochastic differential games describe a control problem with a large number $N$ of agents by letting $N$ tend to infinity. The main idea of mean-field game theory is to designate an approximate form of equilibrium between the symmetric agents' strategies while the number of agents goes to infinity.
The impact of the individual decisions of the other agents is becoming extremely weak compared to the overall impact as $N$ increases. The limiting model emerges from the fact that the dynamics of the individual players decouple and each player constructs her strategy from her own state and from the state of the mass of an infinite number of co-agents of hers which is called the mean-field approach.
\vskip 0.1cm
The mean-field approach in the context of differential games has been independently developed in three mathematical settings. The term mean-field games was introduced By J.-M. Lasry and P.-L. Lions in a series of papers, see \cite{OLL} and \cite{LL}
and the references therein, using nonlinear PDE's.  Independently M. Huang, P. Caines,
Malham\'{e} developed a similar approach, see \cite{HCM2} in the general setting of stochastic processes. They use the term Nash Certainty Equivalence for constructing the mean-field games via a converging iteration of well defined random dynamics and control problems. The contributions \cite{BF,CA} deepen and extend the results for diffusion processes. V.Kolokotsov et. al. investigate mean-field games theory in the setting non-linear Markov Processes where a $1/N$ estimate appears, see \cite{K2}. For work on games with discrete state space see Gomes et. al. \cite{GO} and Basna et. al. \cite{BHK}.
\vskip 0.1cm
There are numerous important applications of mean-field games in many areas, of which we only mention a few. Caines and collaborators investigate the application to large communication as well as electricity networks and analyse the behavior of the large population dynamics. The group of researchers collaborating with Lions examine applications for the oil industry and in the analysis of pedestrian crowds. Carmona et al. focus on inter-banking trading.
\vskip 0.1cm
The investigations in this work are carried out in the framework of non-linear Markovian propagators, respectively time inhomogeneous non-linear Feller processes, which was developed by Vassili Kolokoltsov \cite{K3} \cite{K4}. We focus on propagators related to processes of pure jump type with finite intensity measure in a finite dimensional Euclidean space, that can be identified with the possible decisions of the players. This way we generalize results obtained in \cite{BHK}. In analogy to \cite{BHK} our starting point of the so called closed-loop construction including an optimal control is the following backward Kolmogorov equation for $N$ players:
\begin{eqnarray}\label{eqn:kolmogorov}
 \frac{\prtl f_s}{\prtl s} + \AF^N[s,\xb,\mu^N,u]f_s(\xb)
            &=& 0,\quad 0\le t < s\le T \\
    f_T(\xb)&=&  \Phi({\xb}),  \quad \xb\in\RR^{Nd} ,\nonumber
\end{eqnarray}
where $f$ is a function from the domain $\mathfrak{D}(\AF^N[s,\xb,\mu,u])$,
$\mu^{N}:= {\tfrac{1}{N}}\sum_{i=1}^{N} \delta_{x_i}$ is a normalized sum of Dirac measures in $\RR^d$, and the parameter $u\in\RR^m$ represents a control law. Finally, for $\xb\in\RR^{Nd}$, $N\in\NN$ and $t\le s\le T$ the generator $\AF^N$ is of the form
\begin{equation}\label{jumpgenerator}
 \AF^N[s,\xb,\mu^N,u]f_s(\xb) =  \sum_{i=1}^{N}  \Ab^{N,i}[s,x_i,\mu^N,u]f_s(\xb),
\end{equation}
where the operators
\begin{equation}
   \Ab^{N,i}[s,x_i,\mu^N,u]f_s(\xb)
   := \int_{\RR^d} \left(f_{i'}(s,y)-f_{i'}(s,x_i)\right)
                     \nu(s,x_i,\mu^N,u,dy),
\end{equation}
on $C_\infty(\RR^{Nd})$, referring to distinct indistinguishable agents or players $i$, act on the component $x_i\in\RR^d$ in the argument of the function $f$ only, while all other components of $f$ are kept fix. This leads us to introduce the functions $f_{i'}\in C_\infty (\RR^{d})$ such that $f_{i'}(x_i) = f_{\xb'_i}(x_i)=f(\xb)$, where the vector $\xb'_i\in\RR^{(N-1)d}$ is kept fix, and $\xb'_i$ is derived from the vector $\xb$ by removing the component $x_i$. The particle limit being anticipated, the empirical mean $\mu^N$ is replaced by an external constant measure valued parameter $\rho\in\MM:=\MM(\RR^d)$, which is the set of finite measures, representing the limiting distribution $\mu_t$ in (\ref{16}), which solves the kinetic equation. Moreover, the individual dynamics split and we may identify the operators $\Ab^{N,i}$ with the integral operator
\begin{equation}\label{Jumpoperator}
 \Ab[s,x_i,\rho,u]f_{i'}(x_i) :=  \Ab^{N,i}[s,x_i,\rho,u]f(\xb)
\end{equation}
on $C_\infty(\RR^{d})$. We finally drop the index $i$ alltogether. Adopting a notation from physics we shall say the operator $\AF^N$ describes the dynamics of the $N$-mean-field model.
\vskip 0.1cm
As mentioned above the construction involves a mean-field type limit consistent with a given optimal control problem.
This is a particular example of measure valued limits from the theory of interacting particle systems. A key role within the toolbox of this theory plays the injection from the equivalence class $S\RR^{Nd}$ of vectors in $\RR^{Nd}$, which are identical up to a permutation of players, into the set of point measures in $\RR^d$, defined by

\begin{equation}\label{x2Px}
  \xb = (x_1,\ldots ,x_N)\quad\longrightarrow\quad
  \fracN (\delta_{x_1}+ \ldots + \delta_{x_N})\ .
\end{equation}

More precisely, for arbitrary $N\in\NN$ the mapping constitutes a bijection between $S\RR^{Nd}$ and the subset
$\Pmass^N_\delta:=\{\mu\in\MM \mid \mu=\frac{1}{N} \sum_{k=1}^{N}\delta_{x_k}\}$ of normalized sum of Dirac measures in $\RR^d$ and the spaces may be identified.
For each natural number $N$ the space of symmetric real valued functions on $\RR^{Nd}$ which are invariant under component-wise permutations of their arguments is equivalent with the space of  real valued functions on $S\RR^{Nd}$.

The operator $\AF^N[s,\xb,\mu^N,u]$ on $C_\infty(\RR^{Nd})$ with $\mu^N:=\fracN \delta_{\xb}$ generates a time inhomogeneous jump Feller (Markov) process
${\bf X}^N = (X^{N,1},\ldots ,X^{N,N})$, in $\RR^{Nd}$ and the operator $\Ab[s,x,\mu^N,u]$ generates a time inhomogeneous jump Feller (Markov) process $X$ in $\RR^d$, see Section 2.

Due to the symmetry of the generator $\AF^N$ in (\ref{jumpgenerator}) with respect to permutations of the players, the dynamics for one representative amongst $N$ indistinguishable agents and hence the corresponding time inhomogeneous process is described by $X^{N,1}=X^{N,i}$, $1\le i\le N$, in $\RR^d$. Formally the objective for each of the $N$ players is to find the value function
\begin{equation}
\label{valuefunction-N}
 V^N(t,x;\mub^N_{\geq t})= \sup_{\gamma}  \EXP_{\xb}\!\left[\int_t^T  \!\!\!
      J(s,X^{N,1}_s,\mu^N_s,\gamma(s,X^{N,1}_s;\mub^N_s)) \, ds + V^T(X^{N,1}_T,\mu^N_T) \right]
\end{equation}
on $[0,T]\times \RR^d$, i.e. to maximize her expected payoff over a suitable class of admissible feedback control processes
$\{\gamma(s,X^{N,1}_s;\mub^N_s)\mid \, 0\le s\le T\}\in\mathscr{U}$,
with $\mu^N_s= \fracN \sum_i^N \delta_{X_s^{N,i}}$. Here the cost function
$J: [0,T]\times\RR^d\times\pset_\delta^N\times U \rightarrow \RR$
and the terminal cost function $V^T:\RR^d\times\pset_\delta^N\rightarrow \RR$,
as well as the final time $T$ are given. Suitable conditions on the cost function are given in the main part of the paper.
\vskip 0.1cm
Anticipating the existence of a limiting mean-field $\mub=\{\mu_s\in\MM\vert 0\le s\le T\}$, for each $N$, an explicit expression for the value function is derived by dynamic programming as solution of the HJB equation. In fact, for admissible feedback control processes the HJB equation is well posed. Moreover, the solution coincides with the value function and the resulting optimal feedback control function $\hat\gamma^N(t,\xb;\mu_{\ge t})$ is unique for given start value $\xb\in\RR^{Nd}$ and given $\mu_{\ge t}=\{\mu_t\in\MM\vert t\le s\le T\}$. In addition, Regularity in the parameters is shown, See Section 4 below.

As the number $N$ of players tends to infinity the dynamics of the representative player depends on her own state and the overall respectively limiting distribution of the other players only.
Consequently the limiting kinetic equation for the mean-field with a particular choice of the control law is motivated by the weak form of the one player evolution which reads:
\begin{equation}\label{16*}
 (g,\frac{d}{ds} \mu_s)_{\RR^d}
   = (\Ab[s,\mu_s,\varphi(s)]g,\mu_s)_{\RR^d},
            \qquad\quad
  \mu(0) = \mu,
\end{equation}
for arbitrary $g\in C_\infty(\RR^d)$ and a finite measure $\mu_s\in\MM$ as solution
for $0\le s\le T$. The Ansatz will be verified by the limiting procedure at the last section of the paper. A proof for the uncontrolled system maybe found in \cite{O}.

Existence, uniqueness, and regularity of a non-linear flow to the kinetic equation is shown. Moreover, the Koopman-type propagator to the non-linear flow is investigated and exploited to estimate the difference between the non-linear flow given by the kinetic equation and a representation of the linear $N$-mean-field flow on $\prob^N_\delta\subset\MM$ given by the identification (\ref{x2Px}). The construction exhibits the order of convergence to depend on the regularity of the kernel and the dimension of the underlying Euclidean space.

Mean-field game consistency is said to hold when the number of players $N$ goes to infinity, symmetry being granted, if the problem reduces to the dynamic of the states of a decoupled single player and a pair of two coupled equations, the solutions of which leave each other invariant. The first one is the forward kinetic equation, describing the evolution of the distribution of the population. The second one is a backward Hamilton Jacobi Bellman Equation associated with the value function
\begin{equation}\label{V-MF-limit}
 V(t,x;\mu_{\geq t})= \sup_{\gamma}  \EXP_{\xb}\!\left[\int_t^T  \!\!\!
      J(s,X_s,\mu_s,\gamma(s,X_s;\mu_{\ge s})) \, ds + V^T(X_T,\mu_T) \right]
\end{equation}
on $[0,T]\times \RR^d$, where the linear dynamics, represented  by the process $X$ is given by the Kolmogorov backward equation with generator $\Ab[s,x,\mu_s,\gamma(s)]$ with external measure parameter $\{\mu_t\vert 0\le t\le T\}$, which anticipates the mean-field. The corresponding Markov process $X$ describes the states of one player in the mean-field limit. The value function $V$ stands for the limiting control problem. This is shown by a fixed point argument, the mean-field solving (\ref{16}) and the optimal feedback control being the fixpoints.
In analogy to the control problem for $N$ players, there exists a unique optimal feedback control law $\hat\gamma$ for the optimal payoff $V$ while $\mu_t$ is an external parameter. For the proof of an approximate Nash equilibrium a tagged player with a different preference is introduced while using the optimal control $\hat{\gamma}$ given by the fixpoint theorem otherwise. The difference of the value functions $V^N$
and $V$ for modified model with the additional tagged player reveals the bound $\varepsilon$ for the approximate Nash equilibrium.
\vskip 0.1cm
We conclude the introduction with an overview of how the paper is organized. In Section~2 the dynamics of the game is introduced, in particular the Markovian propagator or time inhomogeneous semi group and the continuous in time Markov processes for one representative player. In Section 3 the $N$ mean-field as well as the limiting dynamics and the generator of the corresponding Koopman-type  propagator are set up. The sensitivity analysis for the associated $N$ mean-field control problems is discussed in Section 4. In the subsequent Section 5
the limit when the number of players tends to infinity is investigated.
Bounds for the approximation error are derived for the dynamics as well
as for the value functions. Finally, the $\varepsilon$-Nash equilibrium is established.


\section{Preliminary Results}

A family of mappings $U^{r,s}$, $T_0\le r\le s\le T$, in a set $S$ is called a propagator in $S$ if $U^{s,s} = id_S$ for all $s\in[T_0,T]$ and the following cocycle property holds:
\begin{equation}
\label{Transition}
  U^{t,s} = U^{t,r}U^{r,s}
\end{equation}
for $T_0\le t\le r\le s\le T$. Here $U^{t,r}U^{r,s}$ is to be interpreted as the composition of mappings.

For linear propagators, respectively evolutions see \cite{VC} Chapter 2, it means the application of linear operators, see \cite{DA} Chapter 3 and \cite{VC} Chapter 2.
In the particular case when the propagator is given by a non-autonomous
jump Feller process $X$ on $\RR^d$, i.e.
\begin{equation}
U^{t,s} f(x):= \EE{f(X_s) | X_t =x}, \quad t\leq s,\quad t,s\in \RR,
\end{equation}
it is well defined on the space of continuous and bounded functions $C_b(\RR^d)$. Moreover, the linear operators $U^{t,r}$ are positivity preserving and satisfy $U^{t,r}1=1 $, $U^{t,t}=id$.

\begin{proposition}
\label{aa}
Let $D$ and $B$, $D\subset B$ be two Banach spaces equipped with a continuous inclusion $D \rightarrow B$ and let
$L_i, i=1,2, t\geq 0,$ be two families of bounded linear operators, which are continuous in time $t$. Assume moreover, that $U^{t,r}_i$ are two propagators in $B$ generated by $L_i, i=1,2$, respectively, i.e. satisfying
\begin{equation}\label{eqn:1a}
\frac{d}{ds}U^{t,s}_i f= U^{t,s}_i L_{i,s} f, \quad
\frac{d}{ds}U^{s,r}_i f=-L_{i,s} U^{s,r}_i f, \ \ t\leq s\leq r,
\end{equation}
for any $f \in D$, which satisfy  $\Norm{U^{t,r}_i}_B\leq c_1, i=1,2$.
Moreover, let $D$ be invariant under $U^{t,s}_1$ and $\Norm{U^{t,s}_1}_D \le c_2$
Then we have \\
i)
\begin{equation*}
U^{t,r}_2-U^{t,r}_1=\int^r_t{U^{t,s}_2(L_s^{2}-L_s^{1})}U^{s,r}_1ds,
\end{equation*}

ii)

\begin{equation*}
\label{1}
\left\|U^{t,r}_2-U^{t,r}_1\right\|_{D\rightarrow B} \leq c_1 c_2 (r-t) \sup \limits_{t\leq s\leq r}\left\|L_s^{2}-L_s^{1}\right\|_{D\rightarrow B}.
\end{equation*}
\end{proposition}
For a proof see e.g. \cite{K2}.
\vskip 0.2cm
Although canonical we need to specify the topological function spaces we shall be working with. We proceed with the Banach space of real valued continuous functions that vanish at infinity $C_\infty (\RR^d)$, equipped with the norm
$\Norm {f}_\infty= \sup_{x\in\RR^d}\norm{f(x)}$. We moreover denote by $C_{\infty}^1(\RR^d)$ the space of continuously differentiable functions $f\in C_\infty (\RR^d)$ such that the derivative $f^\prime$ belongs to $C_{\infty}(\RR^d)$ and equip it with the norm
$\Norm{f}_{1,\infty}:= \Norm{f}_{C_{\infty}^1(\RR^d)}
                    :=\sup_{x\in\RR^d} \norm{f(x)+f^\prime (x)}$.
The Banach space $(C_{\infty}^1(\RR^d), \Norm{\cdot}_{1,\infty})$
can be continuously imbedded in the larger Banach space
$(C_{\infty}(\RR^d), \Norm{\cdot}_{\infty})$. We finally introduce the Banach space $(C_{\infty}^2(\RR^d),\Norm{\cdot}_{2,\infty})$ of twice continuously differentiable functions $f\in C_\infty (\RR^d)$ such that the first derivative $f^\prime$ and the second derivative
$f^{\prime \prime}$ belong to $C_{\infty}(\RR^d)$, where
$\Norm{f}_{2,\infty}:= \Norm{f}_{C_{\infty}^2(\RR^d)}
  :=\sup_{x}\left( \norm{f(x)} + \norm{f^\prime (x)} + \norm{f^{\prime\prime}(x)} \right)$.
\vskip 0.2cm

Specific to non-linear Markov processes and mean-field theory are topological measure spaces. When equipped with the norm
$\Norm{\mu}^{*}= \sup_{\Norm{f}_\infty \leq 1} \norm{(f,\mu)}$
the space $\MM$ of finite measures on $\RR^d$ constitutes a Banach space see \cite{T} Chapter 9, where
$(f,\mu)= \int{f(x) \mu(dx)}$ represents the duality between the two spaces. We consider three subsets of
$\MM$, the unit ball $\MM_1:=\MM_1(\RR^d)$, the set
$\pset:=\pset(\RR^d)$ of all probability measures on $\RR^d$, and
$\pset^N_\delta$. We endow them with the induced topology of $\MM$. We conclude this paragraph by introducing the set of continuous measure valued functions $C([0,T],\,\MM)$, respectively, $C_\mu=\{\mub\in C([0,T],\,\MM)\mid \mu_0=\mu\}$ with $\mub:=\{\mu_s \mid 0\le s\le T \}$ for later purposes. These are also called measure valued curves.

Next we consider real valued functions or functionals on $\MM$ and $\MM_1$, respectively. In this section we only introduce the Banach space $(C(\MM_1),\Norm{\cdot}_{C(\MM_1)})$ of continuous functions  with $\Norm{F}_{C(\MM_1)}:=\sup _{\mu \in \MM_1} \norm{F(\mu)}$. Moreover, we specify the differentiation rule for functionals on $\MM$.
\begin{definition}
A function $F$ on the space $\MM$ of bounded measures on $\RR^d$
is said to be differentiable at $y\in\MM$ in the direction of $\delta_x$, $x\in\RR^d$, if the variational derivative $\deltab_{[y;x]} F$ (or $\deltab_x F (y)$) of $F$ exists, that is the G\^ateaux derivative ($D_{[y;\delta_x]}$) of $F$ in the direction of $\delta_x$, $x\in\RR^d$:
\[
  \deltab_{[y;x]} F := \lim \limits_{s\rightarrow 0} \frac{F(y+s \delta_x)-F(y)}{s}\ .
\]
\end{definition}
Further topologies in the space of functionals on $\MM_1$ are given in the beginning of Section 3.

We conclude the collection of definitions and notions in this section with the set of control laws $\mathscr{U}:= \{\ub: [0,T]\rightarrow \RR^m \mid u_s\in U \mbox{ for all } s \in [0,T]\}$ where the set $U$ is compact and convex having a $C^2$ boundary.

\vskip 0.5cm

{\bf Hypothesis A}
\vskip 0.5cm

For the kernel $\nu(s, x, \rho, u, dy)$ on $[0,T]\times\RR^d\times\MM_1\times\RR^m\times\MM$ we assume that:

\begin{itemize}
\item(A1) the kernel $\nu$ is positive uniformly bounded on $[0,T]\times\RR^d\times\MM_1\times\RR^m\times\MM$ satisfying
    $\nu(s, x, \rho, u, \{x\})=0$;
\item(A2) the kernel $\nu$ is continuous in t;
\item(A3) the kernel $\nu$ is uniformly Lipshitz continuous in $x$, $\mu$, and $u$ with respect to the topologies specified above;
\item(A4) The second order variational derivative $\deltab^2_{[\rho,h]}\nu(t,x,\rho,u)$   is assumed to exist in $C_\infty$ as a function of $h\in\RR^d$.
\end{itemize}
As a uniformly bounded kernel it is uniformly Lipshitz continuous in $x$, $\mu$, and $u$ with respect to the topologies specified above.
Further regularity conditions required for optimizing preferences are summarized in Section 4 below.

Solutions to the forward Kolmogorov equation associated with the adjoint operator ${\AF^N}^{*}$ to the operator (\ref{jumpgenerator}) can be studied in a strong way or in terms of the associated propagator, see below. Since we are imposing very strong assumptions, uniqueness and regularity of a strong solution follow easiest by exploiting known results from the theory of differential equations on Banach spaces. Optimizing payoffs is wrapped up in (comes as) a stochastic optimal control problem for a Markov process. Again we are fishing for more regularity, however, we need to show that the operator (\ref{jumpgenerator}) generates a Feller propagator, which in particular possesses the Markov property, to retrieve the process.

The strong conditions collected in Hypotheses A and Section 4 insure the existence of Markov processes associated with the generators $\Ab$ and $\AF^N$, respectively, even in the presence of a feedback control law for pure jump-type generators while keeping the measure parameter fixed. Our conditions cover the ones in Pliska (positive $\nu$) \cite{PL} and \cite{BA}. These authors first address the case where the policy, in our case the feedback control, is constant in time and only in a later step introduce the feedback control law.

\begin{proposition}
\label{nonlin-flow}
{\bf\emph{i}}) Let $M$ be an open subset in $\MM_1$ and $U\in\RR^m$ an open bounded control set. Assume that the kernel $\nu(s, x,\rho,u, dy)$ in (\ref{Jumpoperator}) is continuous in $[t_0,T]$ for some $t_0\in\RR$, of type $C^r$ in $x$,  $r\ge 1$, of type $C^p$, $p\ge 1$, in the parameters $\rho\in M$ in G\^ateaux sense, and of type  $C^q$, $q\ge 1$, in $u\in U$. Then so is the unique global flow solving the kinetic equation (\ref{16}).
An analogous statement holds for the initial condition $\mu_0\in\MM_1$. Since equation (\ref{16}) satisfies a linear growth condition the global flow exists on the whole space.

{\bf\emph{ii}}) Let the kernel $\nu(s, x,\rho,u, dy)$ be Lipschitz continuous in $x$ and let the assumptions of part i) be satisfied otherwise. Then the unique global flow solving the kinetic equation is Lipschitz continuous in this parameter and the conclusions of part i) remain valid for the other parameters.

\vskip 0.2cm
\end{proposition}

The proof of part {\bf\emph{i}}) is a direct consequence of Theorems 2 and 7 in \cite{L} CH. XVIII as well Remark 3 which extends the results to differential equations with parameters. Existence of the global flow in Banach spaces in the presence of a linear growth condition is
treated in Amann \cite{AH}, Chapter 2. This reference also treats the case ii) when we relax the regularity and assume Lipschitz continuity, only.

\begin{corollary}
Let $\nu$ in (\ref{jumpgenerator}) be bounded and as in Proposition \ref{nonlin-flow} i) except for the parameter $x$, which satisfies ii), and consider the linear Kolmogorov backward equations corresponding to the generators $\Ab[t,x,\rho,u]$ and $\AF^N[t,x,\rho,u]$ for $N$ players. Then the corresponding flows in $\RR^d$ and $\RR^{Nd}$, respectively will be Lipschitz continuous in $x$.
\end{corollary}

\begin{lemma}
\label{aab}
If Hypothesis A holds,
for any given $\rho\in\MM_1$ and $u\in U$ the bounded time-dependent operators $\Ab[t,x,\rho,u]$ and $\AF^N[t,x,\mu^N,u]$ in \ref{Jumpoperator} and \ref{jumpgenerator} generate Feller Markov backward propagators $\Lambda^{t,s}[x,\rho,u]$ and $\psi_N^{t,s}[x,\rho,u]$ on $C_\infty(\RR^d)$, in particular:
\[
  \Norm{\Lambda^{t,s}}\leq c_3, \quad 0\le t\leq s \le T ,
\]
where $\Norm{\cdot}$ denotes the operator norm.
\end{lemma}

The proof is a direct consequence of \cite{TK} in Theorem 2.2 and Proposition 2.4  of Chapter 4, see also \cite{K4} Chapter 3 and references therein, as will be shown subsequently since the measure $\mu^N$ depends on $x$ and the notion of continuity becomes an issue in contrast to our previous work \cite{BHK}. The generators are conservative, i.e.
$\Ab[t,x,\rho,u] 1= 0  = \AF^N[t,x,\mu^N,u] 1$, and the generators satisfy a)-c) in Theorem 2.2, namely
a) the domains of the bounded operators are dense in $C_\infty(\RR^d)$, in fact $\mathfrak{D}(\Ab[t,x,\rho,u])= C_\infty(\RR^d)= \mathfrak{D}(\AF^N[t,x,\mu^N,u])$; moreover
\vskip 0.1cm
b) the generators satisfy a positive maximum principle, i.e. for $\nu\ge 0$, $f\in C_\infty(\RR^d)$ and $x_0$ such that $f(x_0)= sup_{x}f(x)$ the inequality $\int \nu(t,x,dy)(f(y)-f(x_0))\le 0$ implies for $\lambda > 0$
\begin{displaymath}
 \Norm{\lambda f - \Ab[t,x,\mu^N,u] f}\ge \lambda f(x_0) - \Ab[t,x_0,\mu^N,u] f(x_0) \ge  \lambda f(x_0)
 \ge \lambda \Norm{f},
\end{displaymath}
and analogously for $\AF^N[t,x,\mu^N,u]$; finally
\vskip 0.1cm
c) for $\lambda > \max\{\Norm{\Ab[t,x,\rho,u]},\Norm{\AF^N[t,x,\mu^N,u]} \}$ the operators
$(\lambda -\Ab[t,x,\rho,u])^{-1}$ and $(\lambda -\Ab[t,x,\mu^N,u])^{-1}$
are bounded in $C_\infty(\RR^d)$ hence the ranges of $(\lambda -\Ab[t,x,\rho,u])^{-1}$ and $(\lambda -\Ab[t,x,\mu^N,u])^{-1}$ are dense in $C_\infty(\RR^d)$,
which finishes the proof.
\begin{proposition}
\label{process-N}
Under Hypothesis A and for any given $\rho\in\MM_1$ and $u\in U$, there exist a jump Markov processes generated by $\Ab[t,x,\rho,u]$ and \\ $\AF^N[t,x,\mu^N,u]$ in \ref{Jumpoperator} and \ref{jumpgenerator}, respectively, with sample paths in the space of cadlag functions.
\end{proposition}
For a proof, see \cite{TK}, Theorem 2.7 of Chapter 4.


\section{Propagators in the Space of Bounded Measures}

In the sequel we investigate the evolution of the laws associated with the $N$-player dynamics and the one of the limiting law as the number of players tends to infinity. They are linked by a weak law of large numbers for processes of pure jump type which has been derived by Oelschl\"{a}ger \cite{O} Theorem 2 in the absence of controls. Oelschl\"{a}ger's proof meets the requirement of jump processes, that only have cadlag paths, by choosing a different topology compared to the diffusion case. We proceed by recalling the form of the limiting, so called kinetic equation for constant control parameter.

\begin{remark}
Under the conditions on the initial distributions given in \cite{O}
the laws of the $N$-player processes $X^N$ converge weakly to the Dirac measure concentrated at the solution $\mu_t$ of the integral equation
\begin{displaymath}
 (f,\mu_t)= (f,\mu_0)+\int_0^t (A[s,\cdot,\mu_s,u]f,\mu_s)\, ds
\end{displaymath}
where $f\in C_\infty(\RR^d)$. Moreover $\mu_t$ is the unique deterministic solution of
\begin{equation}
\label{16}
    \frac{d}{dt}\mu_t=\Ab^*[t,\mu_t,u]\mu_t, \ \ \mu_0=\mu, \ \ t \in [0,T]\ ,
\end{equation}
where
\begin{equation}
  \Ab^*[t,\mu_t,u]\mu_t = \int_{\RR^d} \nu(t,y,\mu_t,u,d(x+y)) \mu(dy)
                           - \nu(t,x,\mu_t,u,dy)\mu(dx) \ .
\end{equation}
\end{remark}
For the sake of more comprehensive notation we fix the control parameter and drop it while investigating the limiting dynamic.
\begin{corollary}
\label{b}
Let Hypotheses A be satisfied.\\
{\bf\emph{i}}) Then the non-linear flow $\alpha$ induced by the kinetic equation (\ref{16}) is twice continuously differentiable in $\rho$ in variational sense and Lipschitz continuous  in the control parameter $u$.\\
{\bf\emph{ii}}) Let $\nu$ be uniformly Lipschitz continuous in the measure parameter $\rho$, only. For all $\mu, \eta \in \MM_1$ the unique solution to equation (\ref{16}) given by Proposition \ref{nonlin-flow} is Lipschitz continuous in the initial data i.e:
\begin{equation}
\label{17}
  \Norm{\alpha(0,t,\mu) - \alpha(0,t,\eta)}^{*} \leq
          C(T)\Norm{\mu_0-\eta_0}^*.
\end{equation}
\end{corollary}

\begin{proof} Result {\bf\emph{i}}) is a direct consequence of Proposition \ref{nonlin-flow} with $\nu$
Lipschitz continuous in $x$, $p=2$ and $q=1$.\\

{\bf\emph{ii}}) Equation \ref{16} is equivalent to:
\begin{equation}
\label{16-integral}
\mu_t = \mu_0+ \int_0^t \Ab^*[s,\mu_s] \mu_s ds.
\end{equation}
Using Gronwall's lemma in Banach space we find
\begin{equation}
\label{16-bound}
\left\|\mu_t\right\|^{*} \le \left\|\mu_0\right\|^{*} + \int_0^t \left\|\Ab^*[s,\mu_s]\right\| \left\|\mu_s\right\|^{*} ds
 \le \left\|\mu_0\right\|^{*} e^{cT} \leq \infty.
\end{equation}
For the difference
\begin{displaymath}
\mu_t-\eta_t= \mu_0 -\eta_0 +\int_0^t \Ab^*[s,\mu_s] \mu_s  - \Ab^*[s,\eta_s] \eta_s  ds,
\end{displaymath}
since the uniformly bounded operator $\Ab$ is uniformly Lipschitz continuous in $\rho$, Gronwall's Lemma reveals.

\begin{eqnarray}
\lefteqn{\Norm{\mu_t -\eta_t}^{*} }\nonumber\\
 &\le& \left\|\mu_0 -\eta_0\right\|^{*}
                   +\int_0^t  \left\|\Ab^*[s,\mu_s]
                   -\Ab^*[s,\eta_s]\right\|\Norm{\mu_s}^{*}
+  \left\|\Ab^*[s,\eta_s]\right\|\left\|\mu_s -\eta_s\right\|^{*} ds \nonumber\\
 &\le& \left\|\mu_0 -\eta_0\right\|^{*} + \int_0^t [c_1 \left\|\mu_s\right\|^* +c_2] \left\|\mu_s -\eta_s\right\|^{*} ds \le {\Norm{\mu_0 -\eta_0}^{*}} e^{CT}.
\end{eqnarray}
\end{proof}

We introduce the following topologies for spaces of functionals on the set $\MM_1$ of measures. Let the subsets $C^k(\MM_1)$ of functionals $F$ such that
$\Norm{F}_{C^k (\MM_1)} := \Norm{F}_{C(\MM_1)}
 + \sup_{\mu\in\MM_1}\sum_{\ell=1}^k\Norm{\deltab_{[\mu]}^\ell F}_{C_\infty(\RR^{\ell d})}
$ is finite. They also constitute Banach spaces.

Given Hypotheses A let $\alpha(t,s), \quad 0\leq t \leq s \leq T,$ be the flow in the Banach space $(\MM_1(\RR^d),\Norm{\cdot}^{*})$ guaranteed by Proposition \ref{nonlin-flow}, then we define
\begin{equation}
\label{Koopman}
(\Phi^{t,s} F)(\mu):=F(\alpha(t,s,\mu))
\end{equation}
for all $F\in C^1(\MM_1)$. In the sequel we show that $\Phi^{t,s}$ generalizes the Koopman operator introduced in \cite{LM} as the adjoint of a Frobenius Perron operator on a measure space respectively by \cite{RS} in a Hilbert space setting. We call $\Phi^{t,s}$ the Koopman-type propagator, or shortly Koopman propagator, induced by $\alpha(t,s)$.
\begin{remark}
Other than in $\RR^k$ as in \cite{BHK}, where there the Lebesgue measure is the natural candidate, there is no natural choice which turns $\MM$ into a measure space. Therefore, the Koopman propagator (\ref{Koopman}) and it's properties are specified subsequently.
\end{remark}

\begin{lemma}
 \label{phi_C} Let Hypothesis A be satisfied, then $\Phi^{t,s},\ 0\le t\le s \le T$, has the following properties:
\begin{itemize}
\item[1)] $\Phi^{t,s}$ is a linear propagator in $C(\MM_1)$.
\item[2)] $\Phi^{t,s}$ is a contraction propagator in $(C(\MM_1),\Norm{\cdot}_{C(\MM_1)})$.
\item[3)] The propagator $\Phi^{t,s}$ is strongly continuous with respect to $\Norm{\cdot}_{C(\MM_1)}$.
\end{itemize}
\end{lemma}

\begin{proof}
When inserting the definition one immediately sees that $\Phi^{t,s}$ is linear and well defined on
$C(\MM_1)$ due to Proposition \ref{nonlin-flow} and hence is a linear operator. Exploiting the regularity of
the flow $\alpha$ reveals 1). 2) follows form the following observation
\begin{equation}
\begin{split}
&\Norm{\Phi^{s,t}F}_{C(\MM_1)}= \sup_{\mu \in \MM_1} \norm{\Phi^{s,t}F(\mu)} = \sup_{\mu \in \MM_1} \norm{F(\alpha(s,t,\mu))}\\
& \leq \sup_{\mu \in \MM_1} \norm{F(\mu)}=\Norm{F}_{C(\MM_1)}
\end{split}
\end{equation}
Hence 3) follows.
\end{proof}

\vskip 0.2cm

In order to show that the family of operators $\phi^{t,s}$ in (\ref{Koopman}) is a Koopman-type propagator we will verify that it that it is linear, strongly continuous and that the propagator is a contraction.

\begin{proposition}
\label{bbb}
Under the conditions given in Hypothesis A the family of operators $\phi^{t,s}$ defines a time inhomogeneous propagator in $C^1(\MM_1)$, denoted as Koopman propagator in the sequel, as well as $C^2(\MM_1)$. The following holds:

{\bf\emph{i}}) The Koopman propagator constitutes a family of bounded linear operators
in $C^1(\MM_1)$ and $C^2(\MM_1)$, respectively.

{\bf\emph{ii}}) The generator of the propagator $\phi^{t,s}$ in $C^1(\MM_1)$ is defined by
\begin{equation}
\label{Koopmangenerator}
\mathcal{A}[t,\mu]F(\mu)
  =\int_{\RR^d}{\Ab[t,\mu_t] (\deltab_{[\mu;x]}  F)\mu_t(dx)}
\end{equation}
where $\Ab[t,\mu]$ is given in (\ref{Jumpoperator}).

{\bf\emph{iii}}) The propagator $\phi^{t,s}$ is strongly continuous in $C^1(\MM_1)$ and $C^1(\MM_1)$, respectively.
\end{proposition}

\begin{proof}
i) Due to Lemma \ref{phi_C} the family $\{\phi^{t,s}\}$ is a well defined linear operator on $C(\MM_1)$ and hence on $C^1(\MM_1)$ and $C^2(\MM_1)$.
We continue by showing that for arbitrary $0\le t\le s\le T$ the operator $\phi^{s,t}$ is a bounded operator in the space
$(C^1 (\MM_1),\Norm{\cdot}_{C^1 (\MM_1)})$. For $F \in C^1 (\MM_1)$ we have
\[
 \deltab_{[\mu,x]}(\Phi^{t,s} F)
  =  \int_{\RR^d} (\deltab_{[\alpha(t,s,\mu);z]} F) \deltab_{[x]}\alpha(t,s,\mu(dz)).
\]
Due to the regularity guaranteed by Hypothesis A, duality of
the spaces $C_\infty (\RR^d)$ and $\MM_1$ reveals for all $x\in\RR^d$:
\[
 \Norm{\deltab_{[\mu,x]}(\Phi^{t,s} F)}_{C(\MM_1)}
  \le \Norm{\deltab_{[\alpha(t,s,\mu);z]} F}_{C(\MM_1)}
   \Norm{\deltab_{[x]}\alpha(t,s,\mu)}^{*} <\infty.
\]
where $\sup_{x\in\RR^d}\Norm{\deltab_{[x]}\alpha(t,s,\mu)}^{*}$ by Hypothesis A. In fact we exploit the constant given in Corollary \ref{b}.

Much the same argument holds when showing boundedness of the propagator $\phi^{t,s}$ in the space $(C^2(\MM_1),\Norm{\cdot}_{C^2(\MM)})$. For every $x,y\in\RR^d$ and $F\in C^2(\MM_1) $ explicit calculation reveals
\begin{eqnarray*}
 \deltab^2_{[\mu;x,y]}(\Phi^{s,t}F)
   &=&  \int_{\RR^d} (\deltab^2_{[\alpha(t,s,\mu);z]} F)  \deltab^2_{[x]} \alpha(t,s,\mu(dz))\\
 && \quad +\int_{\RR^{2d}} (\deltab^2_{[\alpha(t,s,\mu);z,r]} F) \deltab_{[x]}\alpha(t,s,\mu(dz))
  \deltab_{[y]}\alpha(t,s,\mu(dr)).
\end{eqnarray*}
A uniform bound for the second order variational derivatives of the solution of (\ref{16}) is guaranteed by Proposition (\ref{nonlin-flow}), which implies that the Koopman propagator leaves the space $C^2 ( \MM_1 )$ invariant. Combining this estimate with Lemma \ref{phi_C} finishes part i) of the proof.

ii) According to Proposition \ref{nonlin-flow} the flow $ \alpha(0,s,\mu)$ associated with
the kinetic equation (\ref{16}) exists for all $s\in[0,T]$  and $\mu\in \MM_1$. By inserting the definition of the Koopman propagator (\ref{Koopman}) into the formal definition of a generator, we find for any $\mu\in\MM_1$
\begin{eqnarray*}
\lefteqn{
 \frac{(\phi^{0,s}F)(\mu)-F(\mu)}{s}
   = \frac{F(\alpha(0,s,\mu))-F(\mu_0)}{s} = \frac{F(\mu_s)-F(\mu)}{s}
 }\\
  &=&\frac 1s \int_0^s\int_{\RR^d} (\deltab_{[\mu_t,x]} F) \dot{\mu}_t(dx) \, dt\
   \le \Norm{A^{*}}\Norm{F}_{C^1 (\MM_1)} \Norm{\mu_t}^{*} < \infty
\end{eqnarray*}
with operator norm $\Norm{\cdot}$. Here we used the chain rule for variational derivatives, see Lemma (F.3) in \cite{K3} and (\ref{16-bound}). Exploiting duality
we conclude that the infinitesimal generator ${\bf\mathcal A}[t,\mu]$ is a bounded operator in
$(C^1 (\MM_1),\Norm{\cdot}_{C^1 (\MM_1)})$ of the form (\ref{Koopmangenerator}).\\
iii) Strong continuity of $\phi^{s,t}$ in $(C^1 (\MM_1),\Norm{\cdot}_{C^1 (\MM_1)})$ and $(C^2 (\MM_1),\Norm{\cdot}_{C^2 (\MM_1)})$ is a direct consequence of ii).

\end{proof}

There exists a non-linear Markov process associated with the kinetic equation (\ref{16}), see \cite{K9} where a similar Martingale approach may be used here. For a definition see \cite{K5}. Different constructions also for pure jump cases may be found in \cite{CL,PL}.

In order to prove the mean-field limit, i.e. estimate errors, we need to unify spaces. This is possible since the factor spaces $S\RR^{Nd}$ and the space of $N$-point measures $\Pmass_\delta^N (\RR^{d})$ may be identified. Consequently we introduce the non-linear operator
\begin{equation}
\label{eqn:NMFgenerator}
 \mathfrak{\hat A}^N[t,\mu^N,u] F(\mu^N):=\AF^N[t,\mu^N,u] f(\xb)
\end{equation}
on $C(\Pmass_\delta^N (\RR^d))$ by identification, where as above the empirical measure $\mu^N=\frac 1N \sum_{i=1}^N \delta_{x_i}\in\Pmass_\delta^N (\RR^d)\subset\MM_1(\RR^d) \subset \MM(\RR^d)$
and $\xb=(x_1,\ldots ,x_N)\in \RR^{Nd} $.

\vskip 0.2cm
\begin{proposition}\label{ANexpansion}
$\nu$ satisfies the Hypothesis A and let $F\in C^2_{\infty}(\Pmass_\delta^N(\RR^d))$ then the operator $\mathfrak{\hat A}^N[t,\mu^N,u]$, with $\mu^N = \fracN\delta_{\xb}$, has the representation
\begin{eqnarray*}
\lefteqn{ \mathfrak{\hat A}^N[t, \mu^N,u] F(\mu^N)
 = \int_{\RR^d} \Ab[t, \mu^N, u] \deltab_{[x]} F(\mu^N) \mu^N(dx)
        + \fracN\int_0^1\! (1-s)
 }\\
 &&\!\!\int_{\RR^{2d}}\!\!
  \left(
 \deltab_{[y]}^2 F (\mu^N+ {\tfrac{s}{N}}(\delta_{x+y}-\delta_{x})),
 (\delta_{x+y}-\delta_{x})^{\otimes 2}
   \right)\nu(t,x, \mu^N, u,dy) \mu^N(dx)ds\ .
\end{eqnarray*}
Here $\deltab_{[y]}$ stands for the variational derivative in the direction $\delta_{x+y}-\delta_{x}$.
\end{proposition}
\begin{proof}
For $Y$ and $Y+\zeta$ such that the whole line
$\{Y+ \theta\zeta\mid 0\le \theta\le 1\}$ is in $\Pmass_\delta^N (\RR^d)$ and $F\in C^2(\Pmass_\delta^N (\RR^d))$
the Taylor theorem in $\MM$ gives the following representation, see \cite{DM} and \cite{K3} Corollary F.2:
\begin{displaymath}
 F(Y+\zeta) - F(Y) = \left(\deltab_{[\zeta]} F(Y),\zeta\right)
+\int_0^1 (1-s) \left(\deltab_{[\zeta]}^2 F(Y+s\zeta),\zeta\otimes\zeta\right)ds\ .
\end{displaymath}

Inserting the Taylor expansion of order 2 under the integral for the choice $Y=\delta_{\xb}/N$ and $\zeta= (\delta_{x_i+y} -\delta_{x_i})/N$ finishes the proof.

We gladly conclude that the first term coincides with the Koopman propagator (\ref{Koopman}) as $N$ converges to $\infty$.
\end{proof}
\begin{remark}\label{psi_N} When restricting the operator $\mathfrak{\hat A}^N$ to $C^2(\Pmass_\delta^N (\RR^d))$ in variational sense the representation in the Proposition \ref{ANexpansion} can be exploited to extend $\mathfrak{\hat A}^N$ to $C^2(\MM)$ in variational sense. This is essential for the proof of the approximate Nash equilibrium since the flow associated with the kinetic equation (\ref{16}) will appear in the argument which is not confined to $C^2(\Pmass^N_\delta(\RR^d))$.

Let $\psi^{t,s}_N$, $0\le t\le s\le T$, denote the corresponding propagator on $C^2(\MM)$.
\end{remark}

\section{Optimal Control for a Representative Player}

In this subsection we shall specify how preferences are optimized in the mean-field game. The mean-field game approach proposes that all players are independent and identically distributed (i.i.d.) and the asymptotic optimal control problem, is of standard type for a single representative agent playing against the overall mass i.e the mean-field. Optimization is based on the principle of dynamic programming and the corresponding HJB equation for the Markov jump processes associated with the generators $\Ab$ and $\AF^N$.

Two types of control problems will be covered: One simplified auxiliary cost function $J$ which does not depend on the state but on the dynamics associated with the mean-field only. The other type where the state of an individual player is introduced who is either one of $N$ indistinguishable players, playing against an $N$-mean-field of all players, or one player, playing against the mean field. This player is subject to the limiting dynamics which is assumed to be given by the operator $\Ab$ in \ref{jumpgenerator}.
\vskip 0.2cm
Before reaching the mean-field limit the measure parameter in the dynamics and the value function is an external parameter even when anticipating the mean-field $\mu_t,\ t\in[0,T]$. Therefore optimization of preferences is pursued by traditional control problems of jump-type with identical cost function though the processes $X^N$ and $X$ differ. Roughly speaking we may control the dynamics of the process by changing its jump intensity dynamically. We introduce the class of feedback control laws $\mathscr{U}$ as the set of all Borel measurable maps $\gamma: [0,T]\times\RR^d \rightarrow U$. To any such $\gamma$ we associate the jump rate
$\nu(t,x,\rho,\gamma(t,x),dy)$. Following Bandini and Fuhrman \cite{BA} and Pliska \cite{PL,PL2} Theorems 3,\, 6 probability spaces and filtrations can be constructed such that $X^N$ and $X$ are Markov processes while including the feedback control. In this section we replace the measure valued parameter $\rho\in\MM_1$ by a family of measure valued curves
$\{\mu_s : s\in [t,T]\}$,
indexed by time $0\le t\le T$.
We emphasize that it is treated as an external parameter depending continuously on time $t$. Hence the assumptions in \cite{BA} and \cite{PL,PL2} hold true. We skip the details of the construction in those papers.

\vskip 0.2cm
Let us introduce the function
\begin{displaymath}
 \Theta(t,x,\rho, u)= J(t,x,\rho,u)+\Ab[t,x,\rho,u]V(t,x,\rho),
\end{displaymath}
on $[0,T] \times \RR^d \times \MM_1 \times U$. Adopting the conditions in \cite{BA,FR} we introduce the following collection of Hypotheses:
\vskip 0.2cm
{\bf Hypothesis B}
\begin{itemize}
\item (B1) Assume that $U \subset \RR^m$ is a compact, convex set with  $C^2$ boundary.
 \item (B2) The cost function $J(t,x,\rho,u)$ is a bounded real valued function, concave in $u$, continuous in $t$, Lipschitz continuous in $x$, and its G\^ateaux derivative in the measure parameter $\rho$ in any direction $y \in \MM_1$ exists for all $x \in \RR^d$ and $u\in U$ and satisfies
\[
  \sup_{\rho\in \MM_1}(\norm{D_{[\rho;y]}J(t,x,\cdot, u)}
  +\norm{V^T(x,\cdot)})
  \leq c_1 \Norm{y}^{*}.
\]
 \item (B3) In addition to the variational differentiability of the bounded operator $\Ab$ we assume further that the operator norm satisfies;
\[
  \sup_{\rho\in \MM_1} \norm{D_{[\rho;y]} \Ab [t,x,\cdot,u]}_{C_{\infty}(\RR^d) \rightarrow C_{\infty}(\RR^d)} \leq c_2 \Norm{y}^{*}.
\]
\item (B4) $J(s,x,\rho,u)$ and $\Ab(s,x,\rho,u)$ are $C^2$ differentiable with respect to the control parameter $u \in U \subset \RR^k$.
\item (B5) $\Theta_{u}(t,x,\rho,u)$ satisfies a Lipschitz condition in $x\in\RR^d$, uniformly in $t$, $\rho$ and $u$.
\item (B6) The absolute value of the eigenvalues of the matrices $\Theta_{uu}$ are bounded above by $\gamma > 0$ and, in particular for all $y \in \MM_1$ and feedback control laws $u:= \gamma(t,x;\rho)$:
\begin{displaymath}
  \sup_{\rho\in \MM_1} \sup_{u\in U}\norm{D_{[\rho;y,y]}^2 J(t,x,\cdot, u)}< c ,
\end{displaymath}
\end{itemize}

The conditions in Hypothesis B are meant to guarantee not only an optimal payoff but also a Lipschitz continuous optimal feedback control. Let us fix $\mub\in C([0,T],\MM_1)$.  We proceed by showing existence of a unique solution to the HJB equation

\begin{eqnarray}
\label{12}
  \frac{\prtl W}{\prtl t}
   +\max_{\gamma}[J(t,x,\mu_{t},\gamma(t,x)) \!\!\!&+&\!\!\! \Ab[t,x,\mu_{t},\gamma(t,x)]W]=0 \\
  W(T,x;\mu_T) &=& V^T(x,\mu_T), \nonumber
\end{eqnarray}
on $[0,T]\times\RR^d\times C([0,T],\MM_1)$ which in a subsequent step will be shown to coincide with the value function of the type given in (\ref{V-MF-limit}).
We proceed by recalling previous results which we adopt to our setting, see \cite{BA,PL}. It is worth mentioning that Pliska \cite{PL} makes use of the theory of non-linear semigroups, developed by Crandell and Liggett \cite{CL}, that directly generalizes to propagators.
\begin{proposition}
\label{existenceW} Let $\mub\in C([0,T],\MM_1)$.\\
i) Under the Hypotheses A and B on the bounded linear operator $\Ab[t,x,\mu_{t},\gamma(t,x)]$ and the cost function $J(t,x,\mu_{ t},\gamma(t,x))$ and the terminal data $V^T$ there exists a unique mild solution $W\in C_b([0,T]\times\RR^d\times C([0,T],\MM_1)$ to the HJB equation (\ref{12}) which is almost surely differentiable with respect to $t$.
\vskip 0.2cm
ii) Under the Hypotheses A and B and positive jump rate $\nu$ there exists a unique bounded Borel measurable function $W$ on $[0,T]\times\RR^d\times C([0,T],\MM_1))$ which solves the HJB equation (\ref{12}), i.e. for all $t\in [0,T]$.
\end{proposition}

Since $\mub\in C([0,T],\MM_1)$ the proof of i) is a direct consequence of Corollary 2.8 \cite{BA} and ii) holds since the assumptions of \cite{PL} are satisfied. For the proof it is actually shown that the second term of (\ref{12}) is Lipschitz continuous in $W$. We comment on relaxing the extra conditions below.

\begin{remark}
The theory of time dependent Markovian semi-groups has recovered the results needed for our work. Moreover, shifting the bounded generator $\Ab$ by a constant in order to make it positive does not change the existence and uniqueness results but only changes the optimal reward by a discounting. We therefore conclude that Proposition \ref{existenceW} also holds when relaxing the condition that $\nu$ is positive and autonomous.
\end{remark}

In view of the previous remarks the assumptions we collected in Hypotheses A and B match those in both works, moreover, combining the results of Proposition \ref{existenceW}, uniqueness reveals:

\begin{corollary} Under the Hypotheses A and B on the bounded linear operator $\Ab[t,x,\mu_{t},\gamma(t,x)]$ and the cost function $J(t,x,\mu_{t},\gamma(t,x))$ and the terminal data $V^T$ there exists a unique solution $W\in C_b([0,T]\times\RR^d\times C([0,T],\MM_1))$ to the HJB equation (\ref{12}).
\end{corollary}

We are now ready to address existence and in a further step uniqueness of an optimal feedback control.
\begin{proposition}
\label{existenceV}
Let the Hypotheses A and B on the bounded linear operator $\Ab[t,x,\mu_t,\gamma(t,x)]$ and the cost function $J(t,x,\mu_{ t},\gamma(t,x))$ and the terminal data $V^T$ be valid.

i) Then the unique solution of the HJB equation (\ref{12}) is equal to the value function V in (\ref{V-MF-limit}), i.e. $W=V$.
\vskip 0.2cm
ii) There exists an optimal feedback control $\hat\gamma= \hat\gamma(t,x;\mub)$ on $[0,T]\times\RR^d\times C([0,T],\MM_1)$, which is given by any function satisfying
\begin{displaymath}
  J(t,x,\mu_{t},\hat\gamma) + \Ab[t,x,\mu_{t},\hat\gamma]V=
 \max_{\gamma=\gamma(t,x)}(J(t,x,\mu_{t},\gamma)+ \Ab[t,x,\mu_{t},\gamma]V) .
\end{displaymath}
\end{proposition}

The result is a direct consequence of \cite{BA} Theorem 2.10, see also \cite{PL}.

\begin{remark}
\label{uniquecontrol}
A sufficient condition for an optimal control law to be unique in case of a maximum payoff problem are convexity of the closed control set $U$ and strict concavity of $\Theta(t,x,\cdot)$, see Theorem 2.4 in \cite{FR}.
\end{remark}

Under even stronger assumptions we retrieve uniform Lipschitz continuity in the starting point as will be shown subsequently.

\begin{proposition} \label{existenceV1} Let $\hat\gamma$ be the unique feedback control on $[0,T]\times\RR^d$ guaranteed by Proposition \ref{existenceV}ii and Remark \ref{uniquecontrol} and suppose that Hypotheses A and B hold. Then $\hat\gamma(t,\cdot)$ satisfies a local Lipschitz condition, uniformly in $t\in[0,T]$.
\end{proposition}

\begin{proof} The proof follows the general line of arguments given in \cite{FR} Lemma 6.3. However, in \cite{FR} the problem was to find a unique minimum, while in our setting we are looking for a unique maximum. Hence the assumption (B6) is slightly different.
\end{proof}

In the following we want to study regularity in the functional parameter $\mub \in C([0,T], \MM_1)$ for the solution of the HJB equation \ref{12} and the unique optimal feedback control $\hat\gamma = \hat\gamma(t,x;\mub)$. To this end we introduce $\mub^1,\mub^2 \in C([0,T], \MM_1)$
and vary between the two measures along $\mub^\alpha:=\mub^1+ \alpha (\mub^2-\mub^1) \in C([0,T], \MM_1)$, $\alpha\in[0,1]$. Due to convex optimization the previous existence, uniqueness and regularity results hold true when replacing $\mub$ by $\mub^\alpha$. For each $\alpha\in[0,1]$ this involves the notations $\Ab_\alpha[t]:=\Ab[t,x,\mu^\alpha,\gamma(t,x)]$, $\Lambda_\alpha^{t,s}:=\Lambda^{t,s}[t,x,\mu^\alpha,\gamma(t,x)]$, all parts of
$\Theta_\alpha:=\Theta[t,x,\mub^\alpha,\gamma(t,x)]$,
\begin{equation}
\label{Functionalparameter}
V_\alpha:=V(t,x;\mub^\alpha) \qquad\mbox{ and }\qquad
 \hat\gamma_\alpha:=\hat\gamma(t,x;\mub^\alpha), \quad \alpha \in [0,1].
\end{equation}
Inserting the functional flow $\mub^\alpha$ into equation (\ref{12}) together with Proposition \ref{existenceV}ii will lead to the following HJB equation:
\begin{equation}
\label{14}
 \frac{\prtl V_\alpha(t,x;\mub^\alpha_t)}{\prtl t}
 +  \Theta(t,x,\mu^\alpha_t,\hat\gamma(t,x))=0 \ .
\end{equation}

Smooth dependence of $V^T(x,\mu_T)$ on the measure valued parameter $\mu_T\in\MM_1$ amounts to the existence of a G\^ateaux derivative with respect to this parameter, hence amounts by definition to dependence on a real parameter $\alpha \in [0,1]$. By assumption $(B2)$, the G\^ateaux derivatives $D_{[\mu_T;\mu^2_T-\mu^1_T]} V^T(x,\cdot)$ exist in $C_{\infty}(\RR^d)$ for each $x \in \RR^d$. Hence
the derivative $\frac{\prtl V_\alpha^T}{\prtl \alpha}(.)$ exist and belong to $C_{\infty}(\RR^d)$. We have:
\begin{equation}
\label{15**}
V^T(x;\mub^2)-V^T(x;\mub^1)
=\int_0^1{\frac{\prtl V^T_\alpha}{\prtl \alpha}(x)d\alpha}
\end{equation}
with $V_\alpha$ as in (\ref{Functionalparameter}).

\vskip 0.2cm

We proceed by showing that the value function $V(t,x;\mub)$ given by Propositions \ref{existenceW} and \ref{existenceV}i) is differentiable and uniformly Lipschitz continuous in $\mub$. In a first step we study smooth dependence of the solution of the HJB equation (\ref{14}) on a real parameter $\alpha.$ 
\begin{lemma}
\label{AlphaDifferentiable}
Suppose the Hypotheses A and B are satisfied, then the solution $V_\alpha$ of the HJB equation (\ref{14}) is differentiable with respect to $\alpha \in [0,1].$
\end{lemma}

\begin{proof}
Let us assume that the maximum point of $V_\alpha$ is attained at point $\hat\gamma_\alpha$. We know that $V_\alpha(t,x)$ is a unique classical solution of equation (\ref{14}) hence it is a mild solution as well and by Duhamel's principle, see \cite{MCR}, it can be represented in the following form:
\begin{equation}
\label{Duhamel}
V_\alpha(t,.)= \Lambda^{t,T}_{\alpha} V^T_\alpha (.)+ \int_t^T {\Lambda^{t,T}_{\alpha}J_\alpha(s,.)ds} \quad \forall t \in [0,T], \quad x \in \RR^d .
\end{equation}
Moreover let $ \ V_{\alpha_i}(.)= V_{\alpha_i}(t,x;\mub^\alpha), \ J_{\alpha_i}(s,.)= J_{\alpha_i}(s,,x,\hat\gamma_{\alpha_i}),$
for any $\alpha_i \in [0,1], \quad i=1,2$.

The operator $\Ab_{\alpha}[t]$ is differentiable in $\alpha$ for each $t \in [0,T]$, and using  Proposition \ref{aa}i), we have that for $\alpha_1,\alpha_2 \in [0,1]$ with $\alpha_1 > \alpha_2$,

\[
 \frac{\Lambda_{\alpha_1}-\Lambda_{\alpha_2}}{\alpha_1-\alpha_2}
 = {\tfrac{1}{\alpha_1-\alpha_2}} \int_t^s\Lambda_{\alpha_1}^{t,r}(\Ab_{\alpha_1}[r]
          -\Ab_{\alpha_2}[r])\Lambda_{\alpha_2}^{r,s} dr,
\]
moreover, due to Hypothsis A the generator $A[t,x,\rho,u]$ is twice differentiable and bounded in the measure parameter $\rho$ to give:
\[
 \Norm{\frac{\partial \Lambda_{\alpha}^{t,s}}{\partial \alpha}}
  \le \! \lim_{\alpha_1 \rightarrow \alpha_2}\!  {\tfrac{1}{\alpha_1-\alpha_2}} \! \int_t^s \!\Norm{\Lambda_{\alpha_1}^{t,r}}  \Norm{(\Ab_{\alpha_1}[r]
      -\Ab_{\alpha_2}[r])}\Norm{\Lambda_{\alpha_2}^{r,s}} dr
  \le K\! \max_{r}\Norm{\frac{\partial \Ab_{\alpha}}{\partial \alpha}}_{\alpha=\alpha_2}
\]

where $K= T \max_{0\le r\le s\le T} \Norm{\Lambda_{\alpha_1}^{t,r}}\Norm{\Lambda_{\alpha_2}^{r,s}}$ and $\Norm{\cdot}$ denotes the operator norm.

Together with the assumptions that the mappings $\alpha \mapsto V^T_{\alpha}(.)$ and $\alpha \mapsto J_\alpha(t,.)$ are both differentiable with respect to the real parameter $\alpha$ for each $t \in [0,T]$ and the derivatives exist in $C_\infty(\RR^d)$, we find that the solution $V_{\alpha}(t,.)$ in (\ref{Duhamel}) is differentiable with respect to the real parameter $\alpha$ with bounded derivative for each for each $t \in [0,T]$.
\end{proof}
\begin{corollary}
\label{Alpha}
As a direct consequence of the Proposition \ref{AlphaDifferentiable} we have that the solution $V_\alpha$ of the HJB equation (\ref{14}) is Lipschitz continuous with respect to $\alpha \in [0,1].$ For a proof and more details see \cite{K1}.
\end{corollary}
Under the assumptions A and B and the definition of the $\Ab_\alpha[t], J_\alpha(t,.)$ and $V_\alpha^T(x)$ respectively, for any
$ \mub^1 ,\mub^2 \in C([0,T],\MM)$ we get from equation (\ref{Functionalparameter}) and Corollary (\ref{Alpha}), by replacing $\alpha_1=1$ and $\alpha_2=0$, the following:

\begin{theorem}
\label{E***}
Under the previous assumptions and conditions we have that for any
$ \mub \in C([0,T],\MM)$ the solution of equation (\ref{14}) is uniformly Lipschitz continuous in $\mub$, i.e. for
 $\mub^1, \mub^2 \in C([0,T],\MM)$, there exist a
constant $k \geq 0$ such that

\begin{equation}
\label{15*}
\sup \limits_{(t,x) \in [0,T]\times \RR^d}
\Norm{V(t,x;\mub^1)-V(t,x;\mub^2)}_{\infty} \leq k \sup \limits_{t \in [0,T]} \Norm{\mu^1-\mu^2}^*\ .
\end{equation}
\end{theorem}

The previous results remain valid if we replace the given deterministic curve $\mub$ by a part of it, $\mu_{\ge t}:= \{\mu_s \vert \mub\in C([0,T],\MM_1), t\le s\le T \}$ with $0\le t\le T$.
\vskip 0.2cm
The unique optimal control $\hat\gamma(t,x,\mu_{\ge t})$ is designated to serve as optimal strategy of an individual player in the construction of mean-field games. In the course of the construction the optimal feedback control needs additional regularity in the parameters.

\begin{theorem}
\label{E*}
Beyond the assumptions of the previous theorems and results assume additionally that the resulting unique optimal control
\[\argmax_\gamma (A[t,\rho,\gamma]V(t,x)+J(t,x,\rho,\gamma))\]
is continuous in $t \in [0,T]$ and Lipschitz continuous in $V$ uniformly with respect to $t,x,\mu$.
Then given a trajectory $ \mub \in C_{\mu}\left([0,T],\MM\right)$ and a final payoff $V^T$, the unique optimal control $\hat{\gamma}=\Gamma(t,x; \mu_{\ge t})$ defined via equations (\ref{12}), is Lipschitz continuous uniformly in $\mub$ i.e, for any $\etab,\mub \in C_{\mu}\left([0,T],\MM\right)$:

\begin{equation}
\sup_{t,x} \!\norm{\Gamma(t,x;\eta_{\ge t}) \!-\Gamma(t,x; \mu_{\ge t})} \!\leq k_1 \!\sup_{s \in [t,T]} \Norm{\eta_s \! - \mu_s}^{*},
\end{equation}
for all $ t \in [0,T], x \in \RR^d$.
\end{theorem}

\begin{proof}
The proof is done by combing the assumptions of this Theorem and the results from Proposition \ref{existenceV1} and Theorem \ref{E***}.
\end{proof}

\begin{remark}\label{consistency}
The choice of the parameter $\mub \in C([0,T],\Pmass(\RR^d))$ with start point $\mu_0 = \mu$ will be determined by a fixed point argument, discussed at the end of this section. Let us denote the resulting unique optimal control by

\begin{equation}
\label{Optimacontrol}
\hat\gamma=\Gamma (t,x;\mu_{\ge t}).
\end{equation}

Theorem \ref{E*} allows us to verify the assumptions of Proposition \ref{nonlin-flow}, which will result in well-posedness of the non-linear kinetic equation problem (\ref{16*}) (or (\ref{16})) where the measure parameter $\rho$ and the control function $\gamma$ are replaced by $\mu_t$ and $\Gamma(t,x;\mu_{\ge t})$ respectively.
\end{remark}

The remaining part of this section is to show the existence of a fixed point in the space of flows of probability measures also denoted consistency condition. For the resulting unique curve in $\Pmass(\RR^d)$, a set of strategies can be derived using the optimization procedures explained above.
Let $C([0,T], \mathbf \Pmass(\RR^d))$ the set of continuous probability measure valued functions. The previous set forms a closed convex subset in $C([0,T], \MM_1)$, because $\Pmass(\RR^d)$ is a closed convex subset in $\MM_1$ see \cite{T} and \cite{BELL}.
To any $\mub \in C([0,T], \mathbf \Pmass(\RR^d))$, one can find the solution of the Hamilton Jacobi Bellman equation

\begin{equation}
 \frac{\prtl V(t,x)}{\prtl t}+\max_{\gamma}\left[J(t,x,\mu_t,\gamma)
 +\Ab[t,\mu_t,\gamma]V(t,x)\right]=0,
\end{equation}

from which one can derive the unique optimal control strategy
$\hat\gamma(t,x;\mu_{\ge t}), \forall s \in [t,T], x \in \RR^d$.
Injecting the feedback optimal control $\hat\gamma$ into the kinetic equation defines the mapping $T: \mub \mapsto \hat\mub$ where

\begin{equation}
	\dot{\hat{\mu_t}}=\Ab^*[t,\hat{\mu_t},\hat\gamma(t,x;\mu_{\ge t})]\hat{\mu_t}. \ \ \hat{\mu_0}=\mu, \ \ t \in [0,T]\ ,
\end{equation}

From previous results of Sections 3 and 4 we have that the mapping $T: \mub \rightarrow \hat{\mub}$ is continuous. By Banach-Alaoglu Theorem we get that the unit ball $\MM_1$ is a compact metrizable space with respect to the weak-topology. Together with compactness of $C([0,T], \MM_1)$, due to Arzela-Ascoli Theorem see \cite{KA}, and inequality (\ref{17}), we obtain that $T$ is a compact operator. One completes the proof via Schauder Fixed Point Theorem.

The mean-field consistency condition is incorporated in the equation
\begin{equation}
	\dot\mu_t=\Ab^*[t,\mu_t,\Gamma(t,x;\mu_{t})]\mu_t. \ \ \mu_0=\mu, \ \ t \in [0,T]\ .
\end{equation}
In our applications the solution of the kinetic equation and the control law constitute a fix point. In this case the regularity of the optimal
feedback control in the measure parameter needs to be sharpened from Lipschitz continuity to differentiability of order two.
The regularization introduced in the Appendix closes this gap.
%
\section{Law of Large Numbers: $\epsilon$-Nash equilibrium}

In Physics and Biology scaling limits and analyzing scaling limits are well established techniques which allow to focus on particular aspects of the system under consideration. Scaling empirical measures by a small parameter $h$ in such a way that the measure $h(\delta_{x_1}+\ldots + \delta_{x_N})$ remains finite when the number $N$ of particles or species tends to infinity and the individual
contribution becomes negligible allows to treat the ensemble as continuously
distributed.

Scaling $k^{th}$-order interactions by $h^{k-1}$ reflects the idea that they are more rare than $k-\ell$ order ones for $1 < \ell < k$ and makes them neither negligible nor overwhelming. This scaling transforms an arbitrary generator
$\Lambda_k$ of a $k^{th}$-order interaction into
\begin{displaymath}
  \Lambda_k^hF(h \delta_{\xb})
  = h^{k-1} \sum_{I \subset \left\{1, \cdots,n\right\},\left|I\right|=k}
  \int_{\XX^k}{}\left[F(h \delta_{\xb} - h\delta_{\xb_I}+h\delta_y)
  -F(h\delta_{\xb})\right] \times P(\delta_{\xb_I};dy)
\end{displaymath}
with positive kernel $P(\delta_{\xb_I};dy)$.
The $N$-mean field limit is a law of large numbers for the first order interactions
given by the $N$-mean field evolutions. For the special case of pure jump type
$N$-mean field evolutions, cf. (\ref{jumpgenerator}), we prove weak convergence to
the solution of the kinetic equation (\ref{16}) by exploiting properties of the
corresponding propagators. The procedure consists of introducing the scale
$h= \frac {1}{N}$ and as explained in Section 3 by unifying space, which is pursued by substituting $f(\xb)$ by $F(\frac 1{\norm{\xb}} \delta_{\xb})$, where $\norm{\xb}$ denotes the length of the vector.
For proving an $1/N$-Nash equilibrium we admit that one agent has a decision rule different from the one of the others and give estimates for the errors with respect to a limiting game.
The property exploited for proving the error estimates and consecutively the $N$-mean field limit is given in Proposition \ref{aa}.
\vskip 0.2cm
Let $\psi_N^{t,s}, t \leq s,$ be the $N$-mean-field propagator as in Remark \ref{psi_N}
and assume that $\phi^{s,t}$ is the Koopman propagator defined in (\ref{Koopman}).
Since $\AF^N$ is a bounded operator, the corresponding propagator $\psi_N^{t,s}$ is bounded as well. Exploiting Proposition \ref{aa}i) we derive an estimate for the deviation of the
propagator $\psi^{t,s}_N$ from the limiting Koopman propagator $\phi^{t,s}$ on a sufficiently rich class of functionals $C^2(\MM_1)$ forming a core
for the limiting generator.	
We first study the unrealistic case of a common initial condition $\mu$.
\vskip 0.1cm
By construction $\psi_N^{t,s}$ and $\phi^{t,s}$ satisfy the assumptions of
Proposition \ref{aa}, which reveals the representation:
\begin{equation*}
\left[(\psi_N^{s,t}-\phi^{s,t})F\right](\mu)
= \int_s^t {\left[\psi_N^{s,t}(\hat{\AF}^N[r,\mu]-\AK[r,\mu])
     \phi^{s,t}F\right](\mu) ds}
\end{equation*}
for $F \in C^2(\MM_1)$ and $\mu \in \Pmass_\delta^N(\RR^d)$, independent of the control parameter $u\in U$.
We continue by estimating
\begin{eqnarray}
\label{30}
&&\hspace{-8mm} \sup_{\mu \in \Pmass_\delta^N(\RR^d)}\norm{\left[(\psi_N^{s,t}-\phi^{s,t})F\right](\mu)}
  \!\leq\! \int_s^t  \Norm {\psi_N^{s,t}}
	             \sup_{\Atop{\mu \in \MM_1}{r\in[0,T]}}\!
                 \norm{(\hat{\AF}^N[r,\mu] -\AK[r,\mu]) \phi^{s,t}F(\mu) }ds\  \nonumber
   \\
 &&\leq\!\frac{(s-t)}{N}\Norm{\psi_N^{s,t}}\Norm{\hat{\AF}^N[r,\mu]-\AK[r,\mu]}
        \Norm{\phi^{s,t}} \Norm{F}_{C^2(\MM_1)} \!
 \leq\! \frac{C(T)}{N} \Norm{F}_{C^2(\MM_1)}
\end{eqnarray}
for $0 \leq t \leq s \leq T,$ and Proposition \ref{ANexpansion} was applied in the last step.
The constant $C(T)$ summarizing the three operator norms and integration with respect to time.

This estimate will in a further step be applied to estimate the order of convergence
in the mean-field limit. The initial conditions are chosen to suit the operators and hence
differ while $N$ changes. In fact we shall assume that the initial conditions satisfy
\begin{equation}
\label{300}
 \mu_0^N=\frac{1}{N}(\delta_{X_{0}^{N,1}}+...+\delta_{X_{0}^{N,N}}),
\end{equation}
with random variables ${X_{0}^{N,1}}, \ldots ,{X_{0}^{N,N}}$ and that they converge to a law $\mu_0 \in\Pmass(\RR^d)$ in such a way
that

\begin{equation}\label{301}
\Norm{\mu_0^N-\mu_0}^* \leq \frac{k_1}{N},
\end{equation}
with a constant $k_1 \geq 0$.
Let $C([0,T]\times \mathcal U, C^2(\MM_1))\subset  C([0,T]\times \MM_1 \times \mathcal U)$
be the subspace of continuous functionals $J(t,\rho,u)$, such that $J(t,.,u) \in C^2(\MM_1)$
for each $t,u$.
\begin{lemma}\label{FF}
Let Hypotheses A, B, and Proposition \ref{ANexpansion} be satisfied. Assume initial conditions $\fracN (\delta_{X_{1,0}^N}+...+\delta_{X_{N,0}^N})$ as in (\ref{301}) and a fixed control parameter $\gamma\in U$. Then the following holds:
i) For $t \in [0,T]$ with arbitrary $T \geq 0$:

\begin{equation*}
\norm{(\psi_{N,\gamma}^{0,t}F)(\mu_0^N)-(\phi^{0,t}_\gamma F)(\mu_0)}
\leq \frac{C(T)}{N} \left(\Norm{F}_{C^2(\MM_1)}+k_1\right)
\end{equation*}
with a constant $C(T)$ independent of $\gamma$;

ii) For $J$ on $[0,T]\times \MM_1\times U$:	
	\begin{equation*}
     \norm{{\int_t^T {\!\!J(s, \mu_{\gamma,s}^N, \gamma)ds}
     \! -\int_t^T {\!\!J(s,\mu_{\gamma,s},\gamma) ds}}}  \\
   \!\!\leq \frac{C(T)}{N} \left(\Norm{J}_{C([0,T]\times \mathcal U, C^2(\MM_1))} + k_1\right)
  \end{equation*}
	where $\mu_{\gamma,t}^N $ is the empirical law specified by the propagator $\psi_{N,\gamma}^{0,t}$ and $\mu_{\gamma,s}$ is the law given by the Koopman propagator $\phi^{0,t}_\gamma$ or equivalently $\mu_{\gamma,s}=\alpha(t,s,\mu_0,\gamma)$. For the $N$-mean field dynamics the notation reads
$(\psi_{N,\gamma}^{0,t}J)(\mu_0)= J(\mu_{\gamma,t}^N)$.
\end{lemma}
\begin{proof}
The bounds hold uniformly for all $\gamma\in U$. The proof follows the line argument of the proof of Theorem 5.2 in \cite{BHK}, where the Euclidean unit ball is replaced by the set $\MM_1$ respectively the underlying spaces and the norms of the functionals $F$ and $J$ are specified by their indices.
\end{proof}

Let us now turn to the construction of the $\epsilon$-Nash equilibrium.
We start with the following definition.
\begin{definition}
A strategy portfolio $\Gamma$ in a game of $N$ agents with payoffs
$V_i(\Gamma), i=1,...,N,$ is called an $\epsilon$- Nash equilibrium if, for each
player $i$ and an acceptable individual strategy $u_i$
\[
 V_i(\Gamma) \geq V_i(\Gamma_{-i},u_i)- \epsilon,
\]
where $(\Gamma_{-i},u_i)$ denotes the profile obtained from $\Gamma$ by
substituting the strategy of player $i$ with $u_i.$
\end{definition}

We finally seek approximate Nash equilibria for $N$-mean field games, $N\in\NN$, i.e. we mean to show that deviating from the overall preference $\gamma$, which will finally be determined by the fixpoint in Section 4, does not improve the payoff apart from an infinitesimal error.
Therefore we introduce an additional player with a strategy $\tilde\gamma$ differing from the $\gamma$ of the remaining players entering the mean-field. We proceed with another auxiliary one player model, in which the $N$-mean-field acts as a single player and the differing preference $\tilde\gamma$ is part of it. This reveals refined estimates when comparing to the one player game with Koopman dynamics and uniform strategy $\gamma$. The dynamics of the $N$-mean field with one differing strategy is given by the generator:
\begin{equation}\label{31}
\hat{\AF}^N \![t,\mu^N\!, \gamma,\tilde\gamma] F(\mu^N\!)
= \left[\hat{\AF}^N \![t,\mu^N\!,\gamma]+\Ab^{N,1}\![t,\mu^N\!,\tilde\gamma (t,.)]
     - \Ab^{N,1}\![t,\mu^N\!,\gamma(t,.)]\right]\!\! F(\mu^N\!)
\end{equation}
where $\hat{\AF}^N $ and $\Ab^{N,1}=\Ab$ were defined in
(\ref{eqn:NMFgenerator}) and (\ref{jumpgenerator}), respectively,
$\mu^N=\frac 1{N} {\delta_{\xb}}$,  $F \in C^2 (\MM_1)$.

\begin{remark}\label{def:psiN}
Let $\psi_{N,\gamma,\tilde{\gamma}}^{0,t}$  the $N$-mean-field propagator on $C^2(\MM_1)$
generated by $\hat{\AF}^N [t,\gamma,\tilde{\gamma}]$.
Since $\hat\AF^N$ is a linear combination of the linear operators $\Ab$  and
$\hat{\AF}^N [t,\gamma]$ it inherits their properties, i.e. $\psi_{N,\gamma,\tilde{\gamma}}^{0,t}$ is a bounded linear Feller propagator
and it is Lipschitz continuous in the initial conditions.
\end{remark}

\begin{lemma}\label{G}
Suppose Hypotheses A and B hold and let $\psi_{N,\gamma,\tilde\gamma}^{0,t}$ and the Koopman propagator $\phi^{0,t}_\gamma$ as above with a class of functions
$\gamma: \RR^+ \times \mathbb\RR^d\rightarrow U,$ which are continuous in the first variable and Lipschitz continuous in the second one. Then we recover the estimates given in Lemma \ref{FF}i), ii).
\end{lemma}

The proof follows the lines of the one of Theorem 6.2 in \cite{BHK} with $M\subset\RR^k$,        $x_0^N, x_0, x_{\gamma,\tilde{\gamma}}\in\RR^k$ being replaced by
$\MM_1, \mu_0, \mu_0^N, x_{\gamma,s}, \mu_{\gamma,\tilde{\gamma},s}, \mu_{\gamma,s}$,
where $\mu_{\gamma,\tilde\gamma,t}^N$ is the empirical law of the process specified by the propagator $\psi_{N,\gamma,\tilde\gamma}^{0,t}$ and $\mu_{\gamma, t}$ is
the solution of the  kinetic equation (\ref{16}) with initial value $\mu_0$. The norms for the different spaces have been introduce above.

In order to construct an approximate Nash equilibrium, let the first player have a differing preference and assume that in this case $J$ depends on the state of a tagged player, her differing strategy, and the empirical mean. Hence we have to look at the pairs $(X^{N,1}_{t},\mu_t^N)$, $N\in \NN$ which refer to a chosen tagged agent and an overall mass.
\vskip 0.2cm
Let $C_\infty^{2,2}(\RR^d \times \MM_1)\subset C_\infty^2(\RR^d \times \MM_1)$ denote the subspace of functionals $F(x,\mu)$ on $\RR^d \times \MM_1$, such that for each $\mu \in \MM_1$, the functional $F(.,\mu) \in C^2_\infty(\RR^d)$ and for each $x \in \RR^d, F(x,.) \in C^2(\MM_1)$. Moreover,
let $C([0,T]\times \mathcal U, C_\infty^{2,2}(\RR^d \times \MM_1))$ be the subspace of continuous functionals $F(t,x,\mu,u) \in C_\infty([0,T]\times \RR^d \times \MM_1 \times \mathcal U)$,
such that for each $t,\mu,u$ the functional $F(t,.,\mu,u) \in C_\infty^2(\RR^d)$ and $ F(t,x,.,u) \in C^2(\MM_1)$ for each $t,x,u$.\\
The generators of the pairs $(X^{N,1}_{t},\mu_t^N)$ of processes are defined on the space $C_\infty^{2,2} (\RR^d\times \MM_1)$ and take the form
\begin{equation}\label{ANtag}
 \hat{\AF}^N_{tag}[t,x_1,\mu^N,\gamma,\tilde{\gamma}]F(x_1,\mu^N)
 :=\left(\Ab^{N,1}[t,\mu^N,\tilde\gamma]
   +\hat{\AF}^N[t,\mu,\gamma,\tilde{\gamma}]\right)F(x_1,\mu^N),
\end{equation}
with $\hat{\AF}^N[t,\mu^N,\gamma,\tilde{\gamma}]$ as in (\ref{31}).

\begin{remark}\label{def:phiNtag}
i) The propagator associated with the generator $\hat{\AF}_{tag}^N$ will be denoted by
$\xi_{N,\gamma,\tilde\gamma}^{0,t}$. It possesses the same properties as the propagators $\Lambda$ and $\psi_{N,\gamma,\tilde\gamma}^{0,t}$ i.e. it is a bounded Feller propagator, Lipschitz continuous in the initial condition.

ii) Let $\phi_{\gamma,tag}^{0,t}$ the propagator generated by the family
\begin{equation}\label{32}
\Ab^{1}[t,\mu,\tilde{\gamma}] + \AK[t,\mu,\gamma]
\end{equation}
on $C_\infty^1 (\RR^d\times \RR^k)$. Since the operator $\Ab^{1}=\Ab$ is bounded and more regular in the parameters than $\mathcal A$, the propagator $\phi^{0,t}_{tag,\gamma}$ inherits the properties of the Koopman propagator $\phi^{0,t}_\gamma$. Here we mention in particular that $\phi_{tag,\gamma}^{0,t}$ is a strongly continuous contraction, Lipschitz continuous in the initial condition.

iii) By inserting (\ref{31}) into the definition and by applying Proposition \ref{ANexpansion} we find:
\begin{equation}\label{33}
 \hat{\AF}^N_{tag}[t,x_1, \mu^N,\gamma,\tilde{\gamma}]F(x_1, \mu^N )
  =  \left(\Ab^{N,1}[t, \mu^N,\tilde{\gamma}]
     + \AK[t,\mu^N,\gamma]\right)F(x_1,\mu^N)+O(\fracN).	
\end{equation}
\end{remark}

For the Kolmogorov equation corresponding to this generator we make the
assumption that the initial conditions
$X_{0}^{N,1}$ converge to $X_{0}^1 \in \RR^d$ as $N \rightarrow \infty$, such that
for $k_2 >0$
\begin{equation}\label{34}
 \norm{X_{0}^{N,1}-X_{0}^1} \leq \frac{k_2}{N}.
\end{equation}
\begin{lemma}\label{H}
Let  Hypotheses A and B and the initial conditions (\ref{300}) and (\ref{34}) hold. Assume $\phi^{0,t}_{tag,\gamma,}$
$\xi^{0,t}_{N,\gamma,\tilde\gamma}$, and the cost function $J$ to be as above. Then the following bounds exist for $t \in [0,T]$, $T > 0$:

i) For $F\in C^{2,2} (\RR^d \times \MM_1)$ we have
  \begin{equation*}
	\norm{(\xi_{N,\gamma,\tilde\gamma}^{0,t} F)(X_{0}^{N,1},\mu_0^N)
     -(\phi^{0,t}_{tag,\gamma}F)(X_{0}^1,\mu_0)}
     \leq \frac{C(T)}{N} \left(\Norm{F}_{C_{\infty}^{2,2}(\RR^d \times\MM_1)}+k_1\right)
  \end{equation*}
  with a constant $C(T)$ not depending on $\tilde{\gamma}$;
\vskip 0.1cm
ii) For $J(t,x,\mu,u) \in C([0,T] \times U,C_{\infty}^{2,2}(\RR^d \times \MM_1)$:
  \begin{equation*}
  \begin{split}
	&\norm{\EE {\int_t^T J(s,X_{\tilde\gamma,s}^{N,1},\mu^N_{s,\gamma,\tilde{\gamma}},
     \tilde{\gamma}(s,X_{\tilde\gamma,s}^{N,1}))ds}
     -\EE {\int_t^T J(s,X^1_{\gamma,s},\mu_{s,\gamma},\gamma(s,X^1_{\gamma,s}))ds} } \\
	&\leq \frac{C(T)}{N} \left((T+k_2)\Norm{J}_{C([0,T] \times U,C^{2,2}_{\infty}(\RR^d \times \MM_1)}+k_1\right),
	\end{split}
  \end{equation*}

where the pair $\left(X_{\tilde\gamma,s}^{N,1},\mu^N_{t,\gamma,\tilde{\gamma}}\right)$ is the
Markov process specified by the propagator
$\xi_{N,\gamma,\tilde\gamma}^{0,t}$, and the component processes $X_{\tilde\gamma,s}^{N,1}$ is corresponds to $\Ab^{N,1}[t,\mu_{t,\gamma},\tilde\gamma]$.
$\mu_{s,\gamma}$ is the solution to the kinetic equation (\ref{16}) with initial
condition $\mu_0$.
\end{lemma}

\begin{proof}
The proof follows the general line of arguments given in Theorem 5.6 of \cite{BHK}. The space $\{1,\ldots,k\}\times M$, where $M$ is a subset of the Euclidean unit ball, is replaced by $\RR^d \times \MM_1$. The appropriate norms are specified above.
\end{proof}
The results of this and the previous two sections, and Theorem \ref{H} in particular are based on a feedback control as parameter, depending on time however. The mean-field game methodology implies that the feedback control depends on the law of the limiting dynamics respectively the solution of the kinetic equation and vice versa. Therefore the feedback control has to be twice continuously differentiable with respect to the measure parameter. The lacking regularity is compensated by regularization as described in the Appendix.

Let the kernels $\nu(t,x,\mu,dy,\gamma(t,x;\mub))$ and $\nu(t,x,\mub^N,dy,\gamma(t,x;\mub^N))$, respectively, as before. However, replace $\gamma$ by the mollified version $\Phi_{\delta}[\gamma_j]$, $1\le j\le m$, defined in the Appendix, which is of $C^2$ type with respect to the measure argument and define the corresponding generators
$\hat{\AF}^N_{tag,\delta}:= \hat{\AF}^N_{tag}[t,x,\mu,dy,\gamma(t,x;(\Phi_{\delta}[\gamma_j])_j)]$ and
$\mathcal A_{tag,\delta}:= \mathcal A_{tag}[t,x,\mu,dy,\gamma(t,x;(\Phi_{\delta}[\gamma_j])_j)] $.
The construction insures that the properties of the corresponding propagators
$\phi^{s,t}_{tag,\delta}, \psi^{s,t}_{N,\delta}$ are preserved.  By Proposition \ref{aa} (ii) we get that
\begin{equation}
\begin{split}
& \Norm{\psi^{s,t}_{N}F-\phi^{s,t}_{tag,\delta}F}\leq (t-s) \sup_{t,s \in [0,T]} \Norm{(\hat{\AF}^N_{tag}-\mathcal A_{tag,\delta})\phi^{s,t}_{tag,\delta}F}\\
& \leq (t-s) \sup_{t,s \in [0,T]}  \left(\Norm{\left(\hat{\AF}^N_{tag}- \hat{\AF}^N_{tag,\delta}\right)\phi^{s,t}_{tag,\delta}F}+\Norm{\left(\hat{\AF}^N_{tag,\delta}-\mathcal A_{tag,\delta}\right) \phi^{s,t}_{tag,\delta}F}\right)
\end{split}
\end{equation}
and
\[
\begin{split}
\Norm{\left(\hat{\AF}^N_{tag}- \hat{\AF}^N_{tag,\delta}\right)\phi^{s,t}_{tag,\delta}F} &\leq C(w)(\epsilon(N)+ \frac{1}{j}+\delta(j+1)^d)\Norm{\phi^{s,t}_{tag,\delta}F}_{bLip} \\
& \leq C(w,t) (\epsilon(N)+ \frac{1}{j}+\delta(j+1)^d) \Norm{F}_{bLip}
\end{split}
\]
where
\[\norm{\gamma(t,x;\mub^N)-\gamma(t,x;\mub)}\leq \epsilon(N)\ . \]
From Proposition (\ref{bbb}) follows that $\Norm{\phi^{s,t}_{tag,\delta}F}_{C^2(\MM_1)}\leq C \Norm{F}_{C^2(\MM_1)}$
in straight forward manner. Furthermore, using (\ref{Converge}) we have that:
\begin{equation*}
\norm{(\phi^{s,t} F)(\mu_0)-(\phi^{s,t}_{tag,\delta}F)(\mu_0)}
\leq \Norm{F}_{C^2(\MM_1)}\delta t C(w,t).
\end{equation*}
Hence, we find:
\[
\Norm{\left(\hat{\AF}^N_{tag,\delta}-\mathcal A_{tag,\delta}\right) \phi^{s,t}_{tag,\delta}} \!\leq \frac{C(w,t)}{N} \Norm{\phi^{s,t}_{tag,\delta}F}_{C^2(\MM_1)} \!\leq \frac{C(w,t)}{N} \Norm{F}_{C^2(\MM_1)} \!\left( 1+\frac{1}{\delta}\right)
\]
Hence by choosing $j=N^{\beta}$ and $\delta=\frac{1}{N^{(1-\beta)}}$ with $\beta=\frac{1}{2+d}$  we will have that the rate of convergence will be of $\frac{1}{N^{1/(2+d)}}+ \epsilon(N)$ order.

\begin{theorem}
Let $ \{\Ab[t, \mu, u] \mid t \geq 0, \mu \in \MM_1, u \in \mathcal{U}\}$ be the family of jump type operators given in (\ref{Jumpoperator}) and $\mu$ be the solution to equation (\ref{16}). Assume the following:
\vskip 0.05cm
i) The kernel $\nu(t,\mu,u_t)$ satisfies Hypotheses A and B.
\vskip 0.05cm
ii) The form of $\max \Theta(t,x,\mu,u)$ is given by (\ref{14}).
\vskip 0.05cm
iii) The terminal function $V^T $ is in $C^{2,2}_\infty(\RR^d \times \MM_1)$.
\vskip 0.05cm
iv) The runing cost function $J \in C([0,T]\times\mathcal U, C_\infty^{2,2}(\RR^d \times \MM_1))$.
\vskip 0.05cm
v) The initial conditions
$\mu_0^N = \fracN (\delta _{X_{1,0}^N} + ... +\delta _{X_{N,0}^N})$ of an $N$
player game converge in $(C^2_\infty(\RR^d))^*$ , as $N \rightarrow \infty $, to a
probability law $\mu_0 \in  \Pmass(\RR^d)$ in a way that (\ref{301}) and
(\ref{34})are satisfied.
\vskip 0.05cm
vi) Let $\mub$ be the flow of probability measures induced by the solution of the kinetic equation (\ref{16}) and assume that for the strategy profile $\Gamma(t,x;\mu_t)$ of feedback form the HJB equation (\ref{14}) is satisfied such that Remark \ref{consistency} is valid, i.e. consistency holds.
\vskip 0.12cm
Then the strategy profile $\Gamma(t, x;\mu_t)$, is an $\epsilon$-Nash equilibrium
in an $N$ player game, with
\[
 \epsilon = (\frac{1}{N^{1/(2+d)}}+ \epsilon(N)) (\Norm{J}_{C([0, T] \times\mathcal{U},
                C^{2,2}_{\infty} (\RR^d\times \MM_1)} + \Norm{V^T}_{C^{2,2}_\infty (\RR^d \times \MM_1) }+1).
\]
\end{theorem}
\begin{proof}
Due to Assumptions i)-iv) the HJB equation of the game with one player and dynamics given by the generator $\Ab$ in (\ref{Jumpoperator}) and mean field $\mu_t$ given by Corollary \ref{b} admits a unique optimal feedback control law $\Gamma=\Gamma(t,x;\mu_t)$. By Remark \ref{consistency} the pair $(\mu_t,\Gamma(t,x;\mu_t))$ satisfies the mean field consistency condition. Then the approximate Nash equilibrium follows from the subsequent chain of inequalities, where we
exploit the estimates derived previously under the assumptions i)-vi). To prove the first inequality we apply Lemma \ref{H} ii). Remark~\ref{consistency} with equation (\ref{Optimacontrol}), i.e. the mean field consistency, establishes the second inequality. The final third inequality is guaranteed by Lemma \ref{G}i). This reads:
\begin{eqnarray*}
&&\hspace{-1.5cm}\lefteqn{V^N (0,x^N_{1,0},\mu_{0}^N,\Gamma)
  = \EXP \int_0^T{J (s,X_{\Gamma,s}^{N,1},\mu^N_{\Gamma,s},\Gamma_s)ds}
  \geq \EXP \int_0^T{J (s,X_{\Gamma,s}^{1},\mu_{\Gamma,s},\Gamma_s)ds} -\epsilon }\\
&&\geq \EXP \int_0^T{J (s,X_{\Gamma,s}^{1},\mu_{\Gamma,\gamma,s},\gamma_s)ds} -\epsilon
  \geq \EXP \int_0^T{J (s,X_{\gamma,s}^{N,1},\mu^N_{\Gamma,\gamma,s},\gamma_s)ds} -2\epsilon \\
&&=V^N (0,x^N_{1,0},\mu_{0}^N,\gamma_s) -2\epsilon
\end{eqnarray*}
with $\epsilon$ given above. Recall that the state dynamics of the first od $N$ players, who is subject to an $N$-mean-field, is described in terms of the process $X^{N,1}$ and that the state dynamics of the individual player linked to the mean field is described by the process $X^{1}$.
It is clear that these estimates hold irrespectively at which time $t \in [0, T]$ the game is started. This completes the proof and the construction of the mean-field game in this paper.

\end{proof}

\vskip 0.4cm
\section{Appendix}

For the final result and the construction of the rate of convergence we still need to show that the optimal control law $\gamma(t,x;\mub)$ is of $C^2$ type in the variational derivative sense with respect to the measure. We extend the result in \cite{K10} from a finite state space to $\RR^d$. In Section 4, we have shown that the resulting unique optimal feedback control law derived from the HJB equation (\ref{12}) is in the space of uniformly
bounded Lipschitz continuous functions $C^{bLip}(\RR^d)$, equipped with the norm $\Norm{f}_{bLip} = \Norm{f}+\Norm{f}_{Lip},$ where $\Norm{f}_{Lip}:= \sup_{x\neq y}\frac{f(x)-f(y)}{\norm{x-y}_1},$ with $l_1$-norm
$\norm{y}_1:=\sum_j\norm{y_j}$. Our aim is to approximate all Lipschitz continuous functions by twice differentiable ones. In a first step, let us define an arbitrary mollifier function $\chi$. By definition this function is compactly supported, non-negative, and infinitely smooth on $\RR$ with
$\int_{\RR} \chi(t) dt =1$. Let us define the function $\phi(y)= \prod_{i=1}^{d} \chi(y_i)$ and also the approximating function
\[
\Phi_{\delta}[f](x)= \int_{\RR^d} \frac{1}{\delta^d} \phi(\frac{y}{\delta}) f(x-y)dy = -\int_{\RR^d} \frac{1}{\delta^d} \phi(\frac{x-y}{\delta}) f(y)dy.
\]
Inserting into the definition directly reveals for any $\delta$:
\begin{equation*}
\label{Converge}
 \norm{\Phi_{\delta}[f](x)-f(x)}
 \leq \int_{\RR^d} \frac{1}{\delta^d} \phi(\frac{y}{\delta}) \norm{f(x-y)-f(x)} dy
\leq d \delta \Norm{f}_{Lip} \int_{\RR} \norm{t}\chi(t)dt.
\end{equation*}
This directly gives $\Norm{\Phi_{\delta}[f]}_{bLip} \leq \Norm{f}_{bLib}$ and
\[
  \partial_{x_j} \Phi_{\delta}[f](x)
  = \frac{-1}{\delta^{d+1}}\!\int_{\RR^d}\!\!  \left(\partial_{z_j}\phi \right)(z)\vert_{z=\frac{x-y}{\delta}} f(y) dy
  = \frac{1}{\delta^{d+1}} \! \int_{\RR^d}\!\! \left(\partial_{z_j}\phi\right)(z)
  \vert_{z=\frac{y}{\delta}} f(x-y) dy.
\]
Together with the first estimate we get $\Norm{\Phi_{\delta}[f]}_{C^1}= \Norm{\Phi_{\delta}[f]}_{bLip} \leq \Norm{f}_{bLib}$ by direct calculation. For the second order derivatives there holds:
\begin{equation*}
\Norm{\Phi_{\delta}[f]}_{C^2}= \Norm{\Phi_{\delta}[f]}_{C^1} + \Norm{\frac{\partial }{\partial x_j}\Phi_{\delta}[f]}_{bLip}
\end{equation*}
Hence we find:
\begin{equation}
\label{WWW}
\Norm{\Phi_{\delta}[f]}_{C^2} \leq \Norm{f}_{bLip} \left( 1+ \frac{1}{\delta} \int_{\RR^d} \norm{\chi(t)^{\prime}} dt \right)
\end{equation}
In the second step we approximate Lipschitz continuous functions with respect to a measure parameter $\mu$ by finite dimensional functionals, which in turn can be approximated by a twice differentiable functions using the above method. Let $F$ be a Lipschitz continuous function o   n $\MM_1$. For
$j \in \mathbb N$ and $ \quad k=(k_1,...,k_d)$ with $k_i \in \left\{0,...,j\right\}$, let $x_k^j=(\frac{M}{j})k,$ be the lattice of $(j+1)^d$ points in $[0,M]^d$ and let the functions $\phi_k^j$ be the collection of $(j+1)^d$ functions on $\mathbb R^d$ given by
\[\phi_k^j(x)= \prod_{i=1}^d {\chi(\frac{j}{M}(x_i-k_i\frac{M}{j}))}, \quad \chi(t)= \left\{
\begin{array}{ll}
1-\norm{t}, &\norm{t} \leq 1, \\
0, & \norm{t}\geq 1.
\end{array}
\right.\]
The functions $\phi_k^j$ are non-negative and, for any $j, \sum_{k=(k_1,...,k_d)}\phi_k^j=1$. An arbitrary point $x \in \RR^d$ can belong to the support of at most $2^d$ functions $\phi_k^j$, that satisfy the following
\begin{equation}
\norm{\phi_k^j(x)-\phi_k^j(y)} \leq \frac{j}{M} \norm{x-y}_1.
\end{equation}
Then we can define the subsequent finite-dimensional projections in the space of functions and measures
\[P_j(f)=\sum_{k=(0,...,0)}^{(j,...,j)} f(x_k^j)\phi_k^j\, \quad P^*_j(\mu)=\sum_{k=(0,...,0)}^{(j,...,j)} (\phi_k^j,\mu)\delta_k^j\ .\]
We claim that the corresponding finite-dimensional projections $F_j(\mu)= F(P^*_j(\mu))$ on $C(\MM_1)$, the space of continuous functions on the unit ball $\MM_1$, converge uniformly to $F(\mu)$ and it forms a dense subset on it.
\begin{lemma}
\label{Projections}
The projection $P_j$ has the following properties
\begin{itemize}
\item [i)] $\Norm{P_j}\leq \Norm{f}$
\item [ii)] $\Norm{P_j f-f}\leq 2^d d \frac{M}{j} \Norm{f}_{Lip}$
\item [iii)] $\Norm{P_j f}_{Lip} \leq 2^{d+1} d \Norm{f}_{Lip}$
\end{itemize}
\end{lemma}
\begin{proof}\vspace{1mm}

\begin{itemize}
\item[i)]$
\Norm{P_j}=\sup_{f} \norm{P_jf}= \sup_{f} \norm{ \sum_k f(x_k^j) \phi_k^j}
\leq \sup_f \sup_x \norm{f(x_k^j) \sum_k \phi_k^j} \leq \Norm{f}$
\item[ii)]$
\Norm{P_j-f}=\sum_k\norm{(f(x_k^j)-f(x))\phi_k^j} \leq 2^d \max \norm{f(x_k^j)-f(x)}
\leq 2^d d \frac{M}{j}\Norm{f}_{Lip}$
\item[iii)]Choose an arbitrary $x,y$. Note that in what is coming the sum is taken over not more than $2^{d+1}$ lattice points, $2^d$ for $x$ and $2^d$ for $y$. Let $k_0$ be one of these points. Hence

\begin{eqnarray*}
\lefteqn{\norm{P_jf(x)-P_jf(y)}}\\
&=&\norm{\sum_{k \neq k_0} [f(x_k^j)\phi_k^j(x)-f(x_k^j)\phi_k^j(y)]+ f(x_{k_0}^j)\left(\sum_{k\neq k_0}\phi_k^j(x)-\sum_{k\neq k_0}\phi_k^j(y)\right)}\\
& =&\norm{\sum_{k\neq k_0}(f(x_k^j)-f(x_{k_0}^j))(\phi_k^j(x)-\phi_k^j(y))} \leq 2^{d+1} \Norm{f}_{Lip} \frac{M}{j} d\frac{j}{M} \norm{x-y}_1,
\end{eqnarray*}

with $\norm{x-y}_1 \leq d \frac{M}{j}.$
\end{itemize}
\end{proof}
From Lemma \ref{Projections} we directly get the following result

\begin{proposition} The function $\phi$ and the projection $P^{*}$ satisfy
\medskip

\begin{tabular}{ll}
  i) $\Norm{P^*_j\mu^1-P^*_j\mu^2} \leq 2^{d+1} d \Norm{\mu_1-\mu_2}^{*}$
 &ii) $\Norm{F_j}_{Lip} \leq 2^{d+1} d \Norm{F}_{Lip}$ \\
 iii) $\Norm{P^{*}_j\mu-\mu} \leq 2^d d \frac{M}{j}$
 &iv) $\Norm{F_j(\mu)-F(\mu)}\leq 2^d d \frac{M}{j} \Norm{F}_{Lip}.$\\
 \end{tabular}
\end{proposition}

Let us write $F_j(\mu)$ as a function of $(j+1)^d$ variables in the following way:
\[F_j(\mu)= f_j(\left\{(\phi_k^j,\mu)\right\}), \quad {\bf a}^j=\left\{ a_k^j=(\phi_k^j,\mu)\right\} .\]
As a consequence we have that:
\[
\norm{f_j({\bf a}^{j,1})-f_j({\bf a}^{j,2})}
\leq C \Norm{F}_{Lip} \Norm{\sum(a_k^{j,1}-a_k^{j,2})\delta_{x_k^j}}_{bLip^*}
\!\leq C \Norm{F}_{Lip}\norm{{\bf a}^{j,1}-{\bf a}^{j,2}}_1,
\]
where $C= 2^{d+1} d$. Therefore $f$ is a Lipschitz continuous function in ${\bf u}$ and the smoothing approximation method above may apply here.
\begin{remark}
The approximation we have presented above can be extended to the set $[-M,M]^d$ and in general to arbitrary compact subsets of $\RR^d$. Consequently, it holds for the whole $\RR^d$.
\end{remark}

\section{Acknowledgment}
The authors are indebted to the unknown referee for structural and technical suggestions, which improved the paper substancially. We deeply enjoyed scientific discussions on control theory and mean-field games with Sigurd Assing, Fabio Bagagiolo, Rainer Buckdahn, Christine Grün, Juan Li, and Marianna Troeva. Rani Basna and Astrid Hilbert gratefully acknowledge financial support by the Royal Swedish Academy of Sciences.

\end{document}